\documentclass[reqno,12pt]{amsart}

\usepackage{latexsym}
\usepackage{amsmath,amsfonts,amssymb,epsfig}
\usepackage{amsthm,amscd}
\usepackage{mathrsfs} 
\usepackage{hyperref}
\usepackage{cancel}
\usepackage{wrapfig}
\usepackage[mathcal]{eucal}
\usepackage{mathtools}
\usepackage{setspace}
\usepackage{accents}
\usepackage{color}

\usepackage[heads=LaTeX]{diagrams}
\diagramstyle[labelstyle=\scriptstyle]

\makeatletter
\renewcommand\l@subsection{\@tocline{2}{0pt}{2pc}{5pc}{}}
\makeatother

\newcounter{mtheorem} % numerical counter 
\newcounter{mcorollary} % numerical counter 

\newtheorem{theorem}{Theorem}
\newtheorem{mtheorem}[mtheorem]{Theorem} % numerical counter
\newtheorem{definition}[theorem]{Definition}
\newtheorem{remark}[theorem]{Remark}
\newtheorem*{remark*}{Remark}

\newtheorem{mcorollary}[mcorollary]{Corollary}
\newtheorem{proposition}[theorem]{Proposition}
\newtheorem{lemma}[theorem]{Lemma}
\newtheorem*{klemma*}{Key Lemma}

\newtheorem{example}[theorem]{Example}

\newlength{\dhatheight} % doublehat
\newcommand{\doublehat}[1]{%
    \settoheight{\dhatheight}{\ensuremath{\hat{#1}}}%
    \addtolength{\dhatheight}{-0.35ex}%
    \hat{\vphantom{\rule{1pt}{\dhatheight}}%
    \smash{\hat{#1}}}}

\newcommand{\no}{\noindent}

\newcommand{\bex}{\begin{example}\em}
\newcommand{\eex}{\end{example}}

\def\R{\mathbb{R}}
\def\Z{\mathbb{Z}}

\graphicspath{{./pict/},{./../pict/}}

\numberwithin{equation}{section} 
\numberwithin{theorem}{section}
\title[Volume-preserving  vector fields and finite type invariants]{On volume-preserving  vector fields and\\ finite type invariants of knots.} 

\date{\today}

\author{R. Komendarczyk}\thanks{The first author acknowledges the support of DARPA YFA N66001-11-1-4132 and NSF DMS grant \#1043009.}

\address{%Department of Mathematics,
Tulane University,
New Orleans, Louisiana 70118 } 
\email{rako@tulane.edu}
\author{I. Voli{\'c}}\thanks{The second author acknowledges the support of NSF DMS grant \#1205786.}

\address{%Department of Mathematics,
Wellesley College, 
Wellesley, MA 02481-8203} 
 \email{ivolic@wellesley.edu} 
 
\subjclass[2010]{Primary: 57M25; Secondary: 37C50, 76W05, 37C15}	 	
\keywords{volume-preserving fields, asymptotic invariants, configuration space integrals, finite type invariants}
 
\begin{document}
%\onehalfspacing
\setcounter{section}{0}

\begin{abstract}
We consider the general nonvanishing, divergence-free  vector fields defined on a domain in  $3$-space and tangent to its boundary. Based on the theory of finite type invariants, we define a family of invariants for such fields, in the style of Arnold's asymptotic linking number. Our approach is based on the configuration space integrals due to Bott and Taubes.
\end{abstract}

\maketitle

\tableofcontents

%\vspace{-1cm}

\section{Introduction}\label{sec:intro}
Suppose we have a volume-preserving vector field $X$ defined in some compact domain $\mathcal{S}$ of $\R^3$ and tangent to its boundary. In the ideal hydrodynamics or magnetohydrodynamics (MHD), c.f.~\cite{Arnold-Khesin-book} for a comprehensive reference, $X$ plays a role of a vorticity field or a magnetic field. Euler equations (in the ideal hydrodynamics or the ideal MHD) tell us that the flow $\phi_X$ of $X$ evolves in time under volume-preserving deformations. Therefore, quantities associated with $\phi_X$ that are invariant under such deformations are of particular interest to these areas of research.

The best known such invariant is the {\em helicity of $X$}, which we will denote by $\mathscr{H}(X)$. It was first discovered by Woltjer in \cite{Woltjer58}.  Its topological nature, i.e.~the connection to the linking number of a pair of closed curves in space, was first observed in the work of Moffatt \cite{Moffatt69} and then fully described by Arnold in \cite{Arnold86}. This paper concerns the existence and properties of other invariants of volume-preserving fields derived in the style of Arnold from the finite type (or Vassiliev) invariants of knots and links \cite{Vassiliev92, Bar-Natan95, Altschuler-Freidel95, Volic07} (see also questions in \cite[Problem 1990--16]{Arnold-Problems} and \cite[p.~176]{Arnold-Khesin-book}).

In more detail, and following the general idea of \cite{Arnold86}, recall that a long piece of an orbit 
 $\mathscr{O}_T(x)$  of a vector field $X$ through $x\in \mathcal{S}$ for time $T$ (or a collection of orbits through different points in $\mathcal{S}$) can be made into a knot (link) by adding a ``short arc'' (or as many short arcs as there are orbits) $\sigma(x,y)$ connecting its endpoints, i.e.
\begin{equation}\label{eq:O-orbits}
 \bar{\mathscr{O}}_T(x)=\mathscr{O}_T(x)\cup \sigma(x,y),\ \ \text{where} \ \ y=\mathscr{O}_T(x)(T). 
\end{equation}
\no Thus for any $T>0$ we obtain a family of knots $\{\bar{\mathscr{O}}_T(x)\}_{x\in\mathcal{S}}$.  Now let $\mathcal{K}$ be the space of knots (the set of embeddings of $S^1$ in $\R^3$ endowed with the $\mathcal{C}^\infty$ topology) and let 
\[
\mathscr{F}\colon \mathcal{K}\longrightarrow\R	 	
\]
be a function, typically a knot invariant. This function can be restricted to the family $\{\bar{\mathscr{O}}_T(x)\}_{x\in\mathcal{S}}$,  resulting in a function 	 	
\begin{align*}	 	
\lambda_{\mathcal{S},T} \colon \mathcal{S} & \longrightarrow \R \\ 	 	
x & \longmapsto \mathscr{F}(\overline{\mathscr{O}}_T(x)).	 	
\end{align*}	 	
This is a prototype for an invariant of $\phi_X$ under smooth isotopies via diffeomorphisms isotopic to the identity. In order to produce an actual numerical invariant of $\phi_X$, and consequently of $X$, we need to remove the dependence on short arcs. For that reason, for some $m>0$ (usually an integer), one considers the limit 
\begin{equation}\label{eq:asymp-value}
\mathscr{F}^m(X)=\lim_{T\to \infty} \int_\mathcal{S} \frac{1}{T^m}\lambda_{\mathcal{S},T}(x) 
\end{equation}
We will call $\mathscr{F}^m(X)$ the {\em asymptotic value of $\mathscr{F}$ along the flow of $X$ (of order $m$)}. Whenever the order $m$ is specified, we may denote $\mathscr{F}^m(X)$ simply by $\mathscr{F}(X)$. If $\mathscr{F}$ is a knot invariant, this usually gives an invariant of $X$ under volume-preserving deformations. In this case, we will refer to $\mathscr{F}(X)$ as an {\em asymptotic invariant of $X$ (of order $m$)}. 

Replacing a single orbit $\overline{\mathscr{O}}_T(x)$ by a collection of $n$ orbits $\{\overline{\mathscr{O}}_T(x_1),\cdots,\overline{\mathscr{O}}_T(x_n)\}$ at distinct points $x_1$,$\cdots$, $x_n$ of $\mathcal{S}$, the above philosophy can be applied to an invariant $\mathscr{F}\colon \mathcal L_n\to \R$, where $\mathcal L_n$ is the space of $n$-component links (defined and topologized analogously to $\mathcal K$). 
  
Arnold showed in \cite{Arnold86} that this technique gives, in the case when $\mathscr{F}$ is the the linking number $\operatorname{lk}$ of pairs of orbits $\{\mathscr{O}(x),\mathscr{O}(y)\}$, a well defined invariant $\mathscr{H}(X)$ which equals the above mentioned Woltjer's helicity. Namely, given a divergence-free field $X$ on $\mathcal{S}$, we have
\begin{equation}\label{eq:helicity}
 \mathscr{H}(X)=\int_{\mathcal{S}\times \mathcal{S}} \Bigr(\lim_{T\to \infty} \frac{1}{T^2} \text{lk}(\bar{\mathscr{O}}_T(x),\bar{\mathscr{O}}_T(y))\Bigl)\, \mu(x)\times \mu(y),  
\end{equation}
where $\mu$ is a volume form on $\R^3$, and the function under the integral is a well-defined $\mu$ almost everywhere integrable function on $\mathcal{S}$. Arnold called $\mathscr{H}(X)$ the average {\em asymptotic linking number} of $X$ and showed that $\mathscr{H}(X)$ is invariant under the volume-preserving deformations of $X$. 

More precisely, let  $\operatorname{Vect}(\mathcal{S},\mu)$ be the Lie algebra of smooth volume-preserving vector fields on $\mathcal{S}\subset \R^3$ equipped with a volume form $\mu$.  Consider the action by the group of smooth volume-preserving diffeomorphisms of $\R^3$ (isotopic to the identity),  $\text{Diff}_0(\R^3,\mu)$:
\begin{align}
\text{Diff}_0(\R^3,\mu)\times \operatorname{Vect}(\mathcal{S},\mu)& \longrightarrow \operatorname{Vect}(g(\mathcal{S}),\mu) \label{eq:SDiff-action} \\
(g, X) & \longmapsto g_\ast X, \notag	 	
\end{align}
where $g_\ast$ stands for the pushforward of the vector field $X$ by the diffeomorphism $g$. Then invariance under the volume-preserving deformations means the invariance under the above action. In other words,
\begin{equation}\label{eq:helicity-invariance}
 \mathscr{H}(X)=\mathscr{H}(g_\ast X).
\end{equation}
\begin{remark*}
{\rm
Observe that $g_\ast X(x)=\frac{d}{dt} g\circ\phi_X(t,g^{-1}(x))\bigl|_{t=0}$.  Thus on the level of flows, the action in \eqref{eq:SDiff-action} maps the flow $\phi_X=\phi_X(t,x)$ of $X$ to the flow $g\circ\phi_X\circ g^{-1}=g\circ \phi_X(t,g^{-1}(x))$ of $g_\ast X$, i.e.
 \begin{equation}\label{eq:flow-conj}
 \phi_X\longrightarrow g\circ\phi_X\circ g^{-1}.
\end{equation}
}
\end{remark*}
In order to state our main results we first need to provide some general information about finite type invariants, leaving further details for Section \ref{sec:conf-integrals} (or see, for example, \cite{Volic07} for a more detailed reference). The basic object in the theory of these invariants is a graded algebra (over any ring, but for us, this will be $\R$) of trivalent diagrams (see Figure \ref{fig:t-diag}) which we will denote by $\mathcal D$.  The subspace of diagrams of degree $n$ consists of those diagrams with $2n$ vertices and is denoted by $\mathcal{D}_n$, where  $k=k(D)$ vertices are on the circle ({\em circle vertices}), and $s=s(D)$ vertices are off the circle ({\em free vertices}).
 Then $\mathcal D$ is the direct sum of $\mathcal{D}_n$ for all $n\geq 1$. %For more precise descriptions, see Definition \ref{def:trivalentdiagrams}. 
For each diagram $D\in \mathcal{D}$, we may construct a function on a knot space $\mathcal{K}$ by means of configuration space integrals, denoted as
\begin{equation}\label{eq:I_D}
I_D:\mathcal K\longrightarrow \R\ .
\end{equation}
%%%%%%%%%%%%%%%%%%%%%%%%%%%%%%%%
\begin{figure}[htbp]
	\centering
		\includegraphics[scale=1.1]{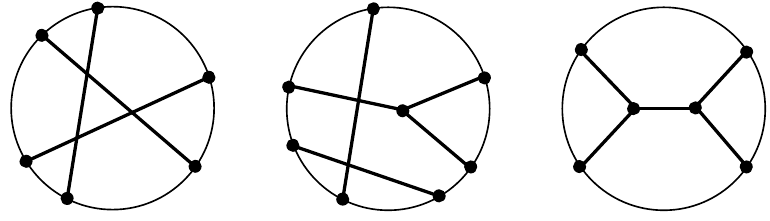}
	\caption{Examples of trivalent diagrams (without labels or edge orientations). The middle diagram is of degree four, while the other two are of degree three.}\label{fig:t-diag} 
\end{figure}
Details about the map $I_D$ are given in Section \ref{sec:conf-integrals}. 

Both $\mathcal D$ and its dual, $\mathcal{W}=\mathcal D^\ast$, called the space of {\em weight systems}, are Hopf algebras.  More formally, any $W\in \mathcal{W}$ is a finite linear combination of diagrams in $\mathcal{D}$. Finite type invariants of knots\footnote{The set up for links is analogous.} are indexed by the subspace of {\em primitive weight systems}, and this is the content of the {\em fundamental theorem of finite type invariants}, originally due to Kontsevich \cite{Kontsevich93}. An alternative proof of this is due to Altschuler and Freidel \cite{Altschuler-Freidel95}, where  the finite type $n$ invariant 	 	
$V_W\colon \mathcal{K}\longrightarrow \R$ associated with the primitive weight system 	 	
\begin{equation}\label{eq:weightsystemsum}	 	
W=\sum_{D\in TD_n} a_{D} D\in\mathcal{W}, \quad a_{D}\in \R,	 	
\end{equation}	 	
is a finite linear combination of functions in \eqref{eq:I_D}:
\begin{equation}\label{eq:V_W-sum}
 V_W=\sum_{D\in TD_n} a_{D}\,I_{D}+b\,I_{D_1},\qquad a_{D},b\in \R.
\end{equation}
\no Here $D_1=\includegraphics[scale=.07]{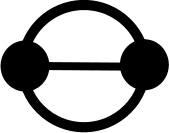}$, and $TD_n$ denotes the set of trivalent diagrams generating $\mathcal{D}$.  For a more precise statement, see Theorem \ref{thm:vassiliev}.
Let us denote the part of the sum $W$ corresponding to diagrams with $k$ vertices on the circle by $W^{k}$.  Thus if $W$ is a degree $n$ weight system, we have $W=\sum^{2n}_{k=1} W^k$, with the top part of $W$ being $W^{2n}$; this corresponds to diagrams all of whose vertices are on the circle (such diagrams are called  \emph{chord diagrams}).  We can then also clearly write 
\begin{equation}\label{eq:V_W-sum-W^k}
V_W=\sum^{2n}_{k=1} V_{W^k}+V_{D_1}.
\end{equation}
We are now ready to state our main result.  
\begin{mtheorem}\label{thm:main-knots}
 Let $X$ be a volume-preserving nonvanishing vector field on a compact domain $\mathcal{S}\subset \R^3$, tangent to the boundary. We then have: 
\begin{itemize}
\item[$(i)$] For any diagram $D\in \mathcal{D}$ of degree $n$, the asymptotic value $\mathcal{I}^{k}_D(X)$, $k=k(D)$ of $I_D$ along the flow of $X$ exists. 
\item[$(ii)$] For any invariant $V_W$ of type $n$, the asymptotic invariant $\mathscr{V}_W(X)$ of order $2n$ exists and equals the asymptotic value $\mathscr{V}^{2n}_W(X)$ of $V_{W^{2n}}$ along the flow $X$. 
\item[$(iii)$] $\mathscr{V}_W(X)$ is invariant under the action by volume preserving diffeomorphisms isotopic to the identity.
\end{itemize} 
\end{mtheorem}
\no Note that, in part $(i)$, $\mathcal{I}^{k}_D(X)$ is not necessarily an invariant because $I_D$ is not one. Further, we may consider a situation where $\mathscr{V}_W(X)=\mathscr{V}^{2n}_W(X)=0$ and see if the lower order averages of $V_W$ exist. For instance, if the asymptotic value $\mathscr{V}^{2n-1}_W(X)$ exists, it may provide a lower order asymptotic invariant 
of $X$. Inductively, if $\mathscr{V}^{j}_W(X)=0$ for $k<j\leq 2n-1$, we may ask if $\mathscr{V}^{k}_W(X)$ defines an invariant of a lower order (in the sense of definition following \eqref{eq:asymp-value}). While we do not answer this question in full generality we obtain the following direct consequence of $(i)$ in Theorem \ref{thm:main-knots} and \eqref{eq:V_W-sum-W^k}.
%%%%%%%%%%%%%%%%%%%%%%%%%%%%%%%%%%%%%%%%%%%%
\begin{mcorollary}\label{cor:main-cor}
Consider a primitive weight system $W$ and suppose for a given $k$ ($k<n$), we have $W^k\neq 0$. Suppose also that the asymptotic value $\mathscr{V}^{j}_{W}(X)$ of $W$ vanishes for every $k<j\leq 2n-1$ as does the asymptotic value $\mathscr{V}^{k}_{W^{k+1}}(X)$. Then the asymptotic invariant $\mathscr{V}_W(X)$ of order $k$ exists and equals the asymptotic value $\mathscr{V}^{k}_{W^{k}}(X)$ of $V_{W^{k}}$ along the flow $X$.
\end{mcorollary}

The meaning of lower order invariants is unclear to us at this point. However, the work in \cite{Kom-helicity09, Kom-helicity10} on asymptotic Brunnian links shows one possible setting where they might appear. 

 A closely related result to Theorem \ref{thm:main-knots} is proven in \cite{Gambaudo-Ghys01} by Gambaudo and Ghys who consider a signature invariant $\sigma\colon \mathcal{K}\longrightarrow \Z$ of knots and its asymptotic counterpart for ergodic volume-preserving fields $X$. In particular, they prove that, in the setting of ergodic fields,  the associated {\em asymptotic signature} $\sigma(X)$ is of order $2$ and satisfies 
\begin{equation}\label{eq:sing=half-helicity}
 \sigma(X)=\frac{1}{2}\mathscr{H}(X).
\end{equation}
%(Remark: In the case of non--ergodic field $\sigma(X)$ an ergodic decomposition of the measure may be used to express $\sigma(X)$ as an integral of averages of $\frac{1}{2}\mathscr{H}_x(X)$ over the ergodic components of the measure corresponding to points $x\in \mathcal{S}$ see \cite{Gambaudo-Ghys01}).

An extension of this work on ergodic fields to other knot invariants appears more recently in the work of Baader \cite{Baader07, Baader11}. In addition, Baader and March\'e \cite{BaaderMarche12} consider asymptotic finite type invariants. The main result of \cite{BaaderMarche12} gives an analog of the identity \eqref{eq:sing=half-helicity} for any asymptotic finite type invariant $\mathscr{V}_W(X)$ of order $n$ whenever $X$ is ergodic and $W$ is degree $n$. Note that Theorem \ref{thm:main-knots} shows that $\mathscr{V}_W(X)=\mathscr{V}^{2n}_W(X)$ is well-defined for a general nonvanishing field $X$ (on a domain $\mathcal{S}$ in $\R^3$), and also indicates a possibility for lower order invariants. Our techniques also lead us to the following counterpart of a result in \cite{BaaderMarche12}.

%
%%%%%%%%%%%%%%%%%%%%%%%%%%%%%%%%%%%%%
\begin{mtheorem}\label{thm:main-ergodic}
 Let $\mu$ be the standard volume form on $\R^3$ and let $X$ be an ergodic $\mu$-preserving nonvanishing vector field on a domain $\mathcal{S}$. Then there exists a singular differential form $\varpi_{W,2n}$ of degree $4n$ on $\mathcal{S}^{2n}$, such that
 \begin{equation}\label{eq:T_W=helicity}
  \mathscr{V}_{W}(X)=c_{W}\,(\mathscr{H}(X))^n=\int_{\mathcal{S}^{2n}} \varpi_{W,2n}\wedge (\underbracket[0.5pt]{\iota_X\mu\times\cdots\times\iota_X\mu}_{n\text{ times}}), 
 \end{equation}
where  $c_{W}$ is a constant independent of $X$, $\iota_X\mu$ is the contraction of $X$ into the form $\mu$, and $\mathscr{H}(X)$ is the helicity defined in \eqref{eq:helicity}. Moreover, the lower order invariants (if they exist) are given as follows
\[
 \mathscr{V}^m_{W}(X)=\int_{\mathcal{S}^{2m}} \varpi_{W,m}\wedge (\underbracket[0.5pt]{\iota_X\mu\times\cdots\times\iota_X\mu}_{m\text{ times}}).
\]  
\end{mtheorem}
Another avenue we explore here are applications to the {\em energy--helicity problem} as considered by Arnold in  \cite{Arnold86} (see also \cite{Arnold-Khesin-book}).  Define the (magnetic) energy of $X$ by 
\begin{equation}\label{eq:L2-energy}
 E(X)=\int_{\mathcal{S}} |X|^2 d\mu,
\end{equation}
i.e.~as the square of the $L^2$--norm of $X$. Consider the problem of minimizing the energy functional $E$ on the orbit $\mathfrak{o}_X=\{g_\ast X\ |\ g \in \text{Diff}_0(\R^3, \mu)\}$ of the action \eqref{eq:SDiff-action} through a fixed vector field $X$. 
If $\mathfrak{o}_X$ is an orbit through a general volume-preserving field $X$ there may  not be a minimizing (smooth) vector field (c.f.~\cite{Freedman99}).
Can the energy be made arbitrary small ? Arnold showed in \cite{Arnold86} that
\begin{equation}\label{eq:energy-helicity-arnold}
E(g^\ast X) \geq C |\mathscr{H}(X)|,
\end{equation}
for any $g\in \text{Diff}_0(\R^3, \mu)$ and for some positive constant $C$ which depends on the ``geometry'' (i.e.~on a choice of the Riemannian metric on $\R^3$). Since $\mathscr{H}(X)$ is invariant under the action \eqref{eq:SDiff-action}, the above inequality
gives a lower bound for the magnetic energy of $X$ along the orbit, whenever $\mathscr{H}(X)\neq 0$. 
Since the bound \eqref{eq:energy-helicity-arnold} is ineffective for vanishing $\mathscr{H}(X)$,  Freedman and He \cite{Freedman-He91} showed a sharper  bound for the $L^{3/2}$--energy\footnote{recall that $L^{2}$--energy majorizes the $L^{3/2}$--energy via the H\"{o}lder inequality.} of $X$ in terms of the {\em asymptotic crossing number}\footnote{denoted in \cite{Freedman-He91} by $c(X,X)$.} $c(X)$ of $X$:
\begin{equation}\label{eq:E-3/2-helicity-bound}
 E_{3/2}(X)\geq \Bigl(\frac{16}{\pi}\Bigr)^{1/4} c(X)^{3/4}\geq \Bigl(\frac{16}{\pi}\Bigr)^{1/4} |\mathscr{H}(X)|^{3/4}.
\end{equation} 
Asymptotic crossing number is not an invariant under the action \eqref{eq:SDiff-action}, but it leads to a topological lower bound for fluid knots, i.e.~divergence-free vector fields constrained to a tube around a knotted core curve $K$ in $3$--space. Namely, denoting by $g(K)$ the genus of $K$, the following estimate is shown in \cite{Freedman-He91}:
\begin{equation}\label{eq:E-3/2-genus-bound}
 E_{3/2}(X)\geq \Bigl(\frac{16}{\pi}\Bigr)^{1/4}\bigl(2 g(K)-1\bigr)^{3/4}\,\text{Flux}(X),
\end{equation} 
where  $\text{Flux}(X)$ is the flux of $X$ through the cross--sectional disk of the tube. 
In Section \ref{sec:q-helicity} of this paper we consider the {\em quadratic helicity} $\mathscr{H}^2(X)$ (recently proposed by Akhmetiev in \cite{Akhmetev12}). Note that $\mathscr{H}^2(X)$ is well defined, thanks to Theorem \ref{thm:main-knots} applied to the square of the linking number\footnote{$\operatorname{lk}^2$, which is the simplest finite type $2$ invariant of $2$--component links}.  Based on the estimate \eqref{eq:E-3/2-helicity-bound} we show 
\begin{mtheorem}\label{thm:q-helicity-energy} 
We have  
\begin{equation}\label{eq:E_3/2-q-helicity}
 E_{3/2}(X)\geq  \Bigl(\frac{16}{\pi}\Bigr)^{1/4} \mathscr{H}^2(X)^{3/8}\geq \Bigl(\frac{16}{\pi}\Bigr)^{1/4} |\mathscr{H}(X)|^{3/4}.
\end{equation} 
\end{mtheorem}
We end this introduction by saying that our techniques are rather different from \cite{Gambaudo-Ghys97, Gambaudo-Ghys01}, where the authors build a ``combinatorial model'' for an ergodic field, and base their considerations on this model. The configuration space integrals have been used by Cantarella and Parsley in \cite{Cantarella-Parsley10} to derive an alternative formula for $\mathscr{H}(X)$ and its ``higher dimensional'' versions.  
Considerations of the current paper are measure--theoretic and in the simplest case can be compared to the work of  Contreras and Iturriaga  on the asymptotic linking number in \cite{Contreras-Iturriaga99}.

Lastly, we wish to indicate that in addition to the results mentioned above, there exists a wealth of approaches  to the problem of defining helicity-style invariants of volume-preserving fields, or more generally measurable foliations; see for example papers  \cite{Akhmetiev05, Verjovsky-Vila94, Bodecker-Hornig04, Evans-Berger92, Laurence-Stredulinsky00, Rivire02, Khesin03, Kotschick-Vogel03} and references given therein.

%%%%%%%%%%%%%%%%%%%%%%%%%%%%%%%%%%%%%%%%%%%%%%%%%
\subsection*{Acknowledgments}  We are grateful to Rob Ghrist, Chris Kottke and Paul Melvin for the email correspondence. The first author thanks the organizers of {\em Entanglement and linking} in Pisa 2011, and in particular Petr Akhmetiev for an interesting conversation  during that meeting.

%%%%%%%%%%%%%%%%%%%%%%%%%%%%%%%%%%%%%%%%%%%%%%%%%%%%%%%%%%%%%
\section{Some metric properties of blowups}\label{sec:blowups}

Before we review configuration space integrals, in this short section we discuss certain properties of blowups needed for later constructions. Throughout this section, $M$ is a smooth compact manifold with corners. We say that $L$ is a {\em submanifold} of a smooth compact manifold with corners whenever it is a $\mathsf{p}$--submanifold in the sense of \cite[Page I.12]{Melrose-book}, which means that local charts come from restriction of the ambient charts to coordinate subspaces. The intersection of two submanifolds $N$ and $L$ is called {\em clean} if and only if it is transverse and $N\cap L$ is a $\mathsf{p}$--submanifold. Recall, following \cite{Bott-Taubes94} and \cite[p. 19]{Thurston99}, 
\begin{definition}\label{def:blowup}
 The {\em blowup of a smooth manifold with corners $M$ along a closed
embedded submanifold with corners $L$} is the manifold with boundary $\operatorname{Bl}(M,L)$ that
is $M$ with $L$ replaced by those points of the unit normal sphere bundle $S(N(L))$ that are actually the images of
paths in $M$. There is a natural smooth map 
\begin{equation}\label{eq:blow-down-map}
\overline{\beta} \colon\operatorname{Bl}(M,L)\longrightarrow  M,
\end{equation}
 called the {\em blowdown map}, and a
partial inverse 
\begin{equation}\label{eq:blow-up-map}
 \beta \colon M-L \longrightarrow \operatorname{Bl}(M,L) - \beta^{-1}(L),
\end{equation}
called the  {\em blowup map}.
\end{definition}
Given a submanifold $N$ of $M$ such that 
$N=\operatorname{cl}(N-L)$ (``$\operatorname{cl}$" denoting the closure), we define, following \cite[Page V.7]{Melrose-book}, the {\em lift} of $N$ to $\operatorname{Bl}(M,L)$ as
\[
 \widetilde{N}=\operatorname{cl}(\beta(N-L)).
\]
Lifting a vector field on $M$ to $\operatorname{Bl}(M,L)$ amounts to lifting the orbits of the flow (c.f.~\cite{Melrose-book}). 
Then we have the following natural fact about lifts given as Proposition 5.7.2 in \cite[Page V.10]{Melrose-book}, which we paraphrase as 
\begin{proposition}\label{prop:lift}
Suppose submanifolds $N$ and $L$ have a clean intersection in $M$. Then the lift $\widetilde{N}$ in $\operatorname{Bl}(M,L)$ is an embedded submanifold of $\operatorname{Bl}(M,L)$ diffeomeorphic to $\operatorname{Bl}(N,N\cap L)$.  
\end{proposition}

\no As a next step we equip $M$ with a smooth Riemannian metric $g_M$ and construct a certain smooth metric $\widetilde{g}_M$ on $\operatorname{Bl}(M,L)$ which agrees with $g$ outside of an $\delta$--tubular neighborhood\footnote{I.e.~the image of $\delta$-disk bundle of $L$ under the normal exponential map.} $U_\delta(L)$ of $L$ and turns $U_\delta(L)-L$ into a ``cylindrical end'' of  $\operatorname{Bl}(M,L)$ as in Figure \ref{F:metric}. More precisely, we define 
\begin{equation}\label{eq:blowup-metric}
\hat{g}_{\operatorname{Bl}(M,L)}=\begin{cases}
dt^2+g_{\partial U_\delta(L)}; & \text{on}\quad (L\times \mathbb{S}^{k-1})\times (0,\delta]\cong U_\delta(L)-L,\\
g_M; & \text{outside of}\quad U_\delta(L).
\end{cases}
\end{equation}
Here $k=\text{codim}(L)$, $t$ parametrizes $(0,\delta]$ segments in $(L\times \mathbb{S}^{k-1})\times (0,\delta]$,
and $g_{\partial(U_\delta(L))}$ is the restriction of $g_M$ to $\partial(U_\delta(L))$. Since $\hat{g}_M$ may not be smooth along $\partial(U_\delta(L))$, we set $\widetilde{g}_{\operatorname{Bl}(M,L)}$ to be  obtained by smoothing $\hat{g}_{\operatorname{Bl}(M,L)}$ in the intermediate region $U_{\frac 54\delta}(L)-U_{\frac 34\delta}(L)$ (see Figure \ref{F:metric}). The above construction will be used later in the case of $C[k;\R^3]$ where $\R^3$ is considered to have the standard metric. 

Next, we indicate a natural estimate which will be very useful in the next section.
%%%%%%%%%%%%%%%%%%%%%%%%%%%%%%%%%%%%%%%%%%%%
 \begin{figure}[htbp]
\begin{center}
 \includegraphics[width=.8\textwidth]{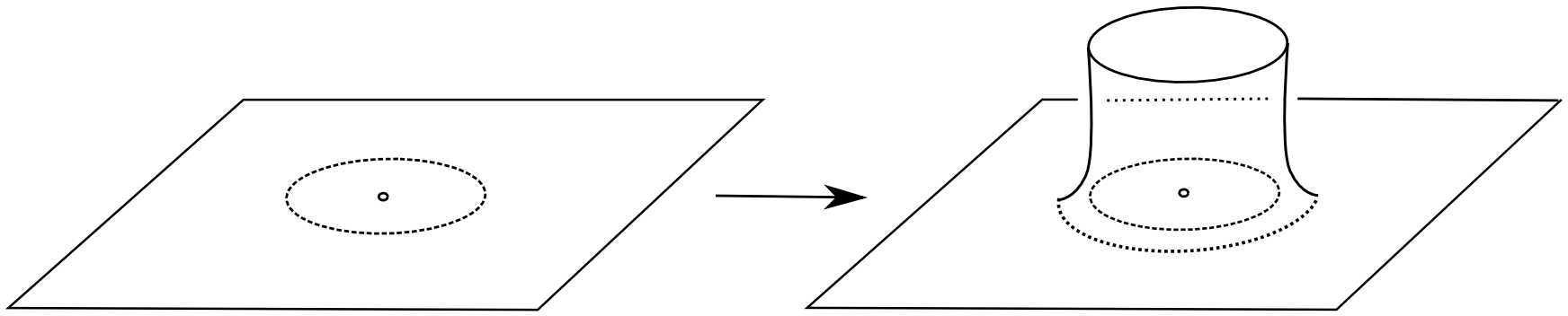}
\caption{Illustration of the metric introduced on the blowup of a point in $\R^2$.}
\label{F:metric}
\end{center}
\end{figure} 
\begin{lemma}\label{lem:A-estimate}
 Let $M$ be a smooth manifold with corners, $L$ a submanifold of $M$, and $\varpi$ a smooth $m$--form on $\operatorname{Bl}(M,L)$. Consider a submanifold $N$ of $M$ whose closure is compact and its lift $\widetilde{N}$ to $\operatorname{Bl}(M,L)$. Define
\begin{equation}\label{eq:A-const}
 A_{\varpi,\widetilde{g}}=\sup_{p\in \widetilde{N}}\ \max_{\substack{v_1,\cdots,v_m\in
T_p \widetilde{N};\\ |v_i|_{\widetilde{g}}=1}} |\varpi(v_1,\cdots, v_m)|.
\end{equation}
Then
\begin{equation}\label{eq:folk-estimate}
 \bigl|\int_N \beta^\ast\varpi\bigr|=\bigl|\int_{\widetilde{N}} \varpi\bigr| \leq  A_{\varpi,\widetilde{g}} \operatorname{vol}(\widetilde{N}).
\end{equation}
\end{lemma}
\no The proof is clear from definitions since $A_{\varpi,\widetilde{g}}$ measures a $C^0$--norm of $\varpi$ along $\widetilde{N}$.

%
%
%%%%%%%%%%%%%%%%%%%%%%%%%%%%%%%%%%%%%%%%%%%%%%%%%%%%%%%%%%%%%%%%%%%%%%
\section{Configuration space integrals}\label{sec:conf-integrals}
 This section contains a brief overview of configuration space integrals (also known as Bott--Taubes integrals). This summary is based on \cite{Volic07} and \cite{Thurston99}.  We also include some technical results about configuration space integrals that will be needed later.  The main result for us is Theorem \ref{thm:vassiliev}.
  Before we describe configuration space integrals, we briefly review the basic notions from the theory of finite type knot invariants.  These invariants have been studied extensively in the last twenty years; for more details, see \cite{Vassiliev92}, \cite{Bar-Natan95} and \cite{Chmutov-Duzhin-Mostovoy-book}. In particular, they are conjectured to \emph{separate} knots. 
 
Let $\mathcal{K}$ be the space of knots, i.e.~smooth embeddings of $S^1$ in $\R^3$, with the $\mathcal C^{\infty}$ topology.  Any knot invariant $V\colon\mathcal{K}\longrightarrow \R$ can be extended to \emph{singular knots}, which are knots except for a finite number of transverse self-intersections, using the \emph{Vassiliev skein relation} given in Figure \ref{F:SkeinRelation}.
%%%%%%%%%%%%%%%%%%%%%%%%%%%%%%%%%%%%%%%%%%%%%%%%%%% 
 \begin{figure}[htbp]
\begin{center}
 \includegraphics[scale=.4]{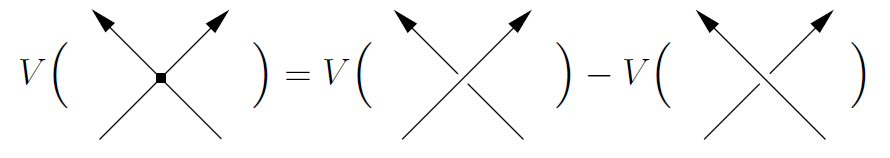}
\caption{Vassiliev skein relation.}
\label{F:SkeinRelation}
\end{center}
\end{figure} 
 The figure is supposed to indicate that all the singularities have been resolved (so a knot with $n$ singularities produces $2^n$ ordinary knots) and $V$ is evaluated on all the resulting knots.
 
 \begin{definition}\label{def:finitetype}
 Invariant $V$ is \emph{finite type $n$} or \emph{Vassiliev of type $n$} if it vanishes on singular knots with $n+1$ singularities.
 \end{definition}
 
  Let $\mathcal V_n$ be the real vector space generated by all type $n$ invariants and let $\mathcal V=\oplus_{n\geq 0}\mathcal V_n$. It is immediate that $\mathcal V_{n-1}\subset\mathcal V_n$, so that one can consider the quotient $\mathcal V_n/\mathcal V_{n-1}$ (which will appear in Theorem \ref{thm:vassiliev}).
  
Finite type invariants are intimately connected to the combinatorics of \emph{trivalent diagrams}. 
\begin{definition}\label{def:trivalentdiagrams}
{\rm
A {\em trivalent diagram} $D$ of degree $n$ is a connected graph consisting of an oriented circle, $k=k(D)$ vertices on the circle ({\em circle vertices}), $s=s(D)$ vertices off the circle ({\em free vertices}), and some number of edges connecting those vertices. The vertex set $\mathcal V(D)$ has cardinality $k+s=2n$, and all vertices are trivalent (the circle adds two to the valence of a circle vertex), from which it follows that the edge set $\mathcal{E}(D)$ is of cardinality $\frac{k+3s}{2}$. The vertices are labeled by the set $\{1,\cdots,2n\}$, and this labeling induces an orientation on the edges in $\mathcal{E}(D)$ (from the lower-labeled end vertex to the higher-labeled one). We will denote by $(i,j)$ the edge connecting vertices $i$ and $j$ where $i<j$. The diagram is regarded up to orientation-preserving diffeomorphisms of the circle. 
}
\end{definition}
Examples of trivalent diagrams (without labels or edge orientations) are presented in Figure \ref{fig:t-diag}. 
Let $TD_n$ denote the set of trivalent diagrams of degree $n$ and let $\mathcal{D}_n$ be the real vector space generated by $TD_n$ modulo subspaces generated by the {\em STU relation} illustrated in Figure \ref{fig:stu-rel}.\footnote{See \cite[p.~3]{Volic07} for more details on the $STU$ relation.}
\begin{figure}[ht]
	\centering
      \includegraphics[scale=.8]{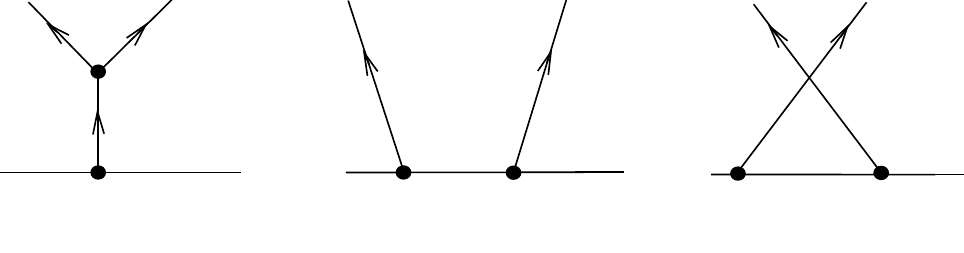}
      \put(-205,0){\small $S$} % STU
      \put(-130,0){\small $-T$}
      \put(-40,0){\small $U$}
      \put(-197,22){\tiny $i$} % S indices
      \put(-202,48){\tiny $j$}
      \put(-103,9){\tiny $i$} % T indices
      \put(-133,9){\tiny $j$}
      \put(-55,9){\tiny $i$} % U indices
      \put(-20,9){\tiny $j$}
	\caption{The $STU$ relation.}\label{fig:stu-rel}
\end{figure}
Vector space $\mathcal{D}=\bigoplus_{n\geq 0} \mathcal{D}_n$ is in fact a commutative and co--commutative Hopf algebra \cite[Theorem 7]{Bar-Natan95}, where the product (and co--product) is 
derived from the operation of connected sum of knots. The dual $\mathcal{W}=\mathcal{D}^\ast$ of $\mathcal{D}$ is known as the space of {\em weight systems}, with $\mathcal{W}_n$ denoting its degree $n$ subspace, i.e.~the dual of $\mathcal{D}_n$. Since $\mathcal{W}$ also has the structure of a Hopf algebra it is sufficient to understand its primitive elements, called
{\em primitive weight systems}.  These generate the entire algebra.  A primitive weight system is one that vanishes on \emph{reducible} diagrams, namely those that are not obtained from two diagrams by connected sum (this informally means that, in an irreducible diagram, one cannot draw a line separating  $\mathcal V(D)$ and  $\mathcal E(D)$ into two nonempty disjoint subsets).

 We now turn our attention to the configuration space integrals.
For a manifold $M$, let $C(q;M)$ be the ordered configuration space of $q$ points in $M$ (i.e.~the $q$--fold product $M^q$, with the thick diagonal removed).  Also recall that that, given a submanifold $N$ of a manifold $M$, the {\em blowup of $M$ along $N$}, $\operatorname{Bl}(M,N)$, is obtained by replacing $N$ by the unit normal bundle of $N$ in $M$ (see Definition \ref{def:blowup}). Finally, for $S$ a subset of $\{1,...,q\}$, let $M^S$ be the product of $|S|$ copies of $M$ in $M^q$, indexed by the elements of $S$, and let $\Delta_S$ be the corresponding (thin) diagonal in $M^S$. 

Now let
$$
A[k;M] =M^k\times \prod_{S\subset\{1, ..., k\},\ |S|\geq 2} \operatorname{Bl}(M^S,\Delta_S).
$$

\begin{definition}\label{D:Compactification}
The \emph{Fulton-MacPherson compactification} of $C(k; M)$, denoted by $C[k;M]$, is the closure of the image of the inclusion
\begin{equation}\label{eq:alpha_M}
\alpha_M \colon C(k;M)\longrightarrow A[k;M],
\end{equation}
%\begin{equation}\label{eq:alpha_M}
%\begin{split}
%\alpha_M : & C(k;M)\longrightarrow A[k;M],\\
% & A[k;M] =M^k\times \prod_{S\subset\{1, ..., k\},\ |S|\geq 2} \operatorname{Bl}(M^S,\Delta_S),
%\end{split}
%\end{equation}
where the $S$--factors of this map are given by the blowup maps\footnote{see Equation \eqref{eq:blow-up-map}}. We denote $\alpha_M$ by $\alpha$ if $M$ is understood, and we will also refer to it as the {\em blowup map} of $C(k,M)$. The {\em blowdown} map $\overline{\alpha}_M:C[k;M]\longrightarrow M^k$ is obtained by the obvious restriction of the projection of $A[k,M]$ onto its $M^k$ factor.
\end{definition}
Equivalently, $C[k;M]$ can be obtained from $M^k$  by successive blowups of $\Delta_S$ diagonals in $M^k$ \cite{Bott-Taubes94, Thurston99}. These blowups have to be performed in the order dictated by the inclusion relation $\subset$ on the indexing sets $S$.  More precisely, if $S'\subset S$, then $\Delta_S$ should be blown up before $\Delta_{S'}$. 
Yet another equivalent definition is due to Sinha \cite{Sinha04}. All these definitions produce diffeomorphic smooth manifolds with corners, compact when $M$ is compact, and homeomorphic to 
a complement of a tubular neighborhood of the thick diagonal in $M^k$. The interior of $C[k;M]$ equals the image of $C(k;M)$ under $\alpha$ and will be denoted by $C_0(k;M)$. For the remainder of this section we will mostly need the case $M=\R^3$.  In this situation, one needs to equip the compactification $C[k;\R^3]$ with a face at infinity for it to be a compact manifold with corners . We also point out that compactification is functorial and in particular we have 
\begin{proposition}[\cite{Fulton-MacPherson94,Sinha04}]\label{prop:conf-funct}
Suppose $g:M\to N$ is an embedding of a smooth manifold $M$ into a smooth manifold $N$. We then  have an induced embedding 
\[
\widetilde{g}:C[k;M]\longrightarrow C[k;N]
\]
 of manifolds with corners, which respects the boundary stratifications and extends the obvious product map $g^k:C(k;M)\longrightarrow C(k;N)$, $g^k=g\times\cdots\times g$, such that the following diagram commutes 
 %%%%%%%%%%%%%%%%%%%%%%%%%%%%
\begin{equation}
\begin{diagram}\label{diag:lifts}
	C[k;M] & &\rTo^{\widetilde{g}} & & C[k;N]  \\
		\uTo_{\alpha_M} & &  & &   \uTo_{\alpha_N} \\
  C(k;M) & &\rTo^{g^k} & & C(k;N).
\end{diagram}
\end{equation}
\end{proposition}
\no The reader may consult, for example, \cite[Corollary 4.8]{Sinha04} for a proof of this proposition.

Given the compactified configuration space $C[q;\R^3]$ and any two positive integers $k$ and $s$, define $C[k,s;\mathcal{K},\R^3]$ to be the pullback bundle in the following diagram
%%%%%%%%%%%%%%%%%%%%%%%%%%%%
\begin{equation}
\begin{diagram}\label{diag:bt}
	C[k,s;\mathcal{K},\R^3] & &\rTo^{p_{k,s}} & & C[k+s;\R^3]  \\
		\dTo_{\bar{\pi}_k} & &  & &   \dTo_{\pi_k} \\
  C[k;S^1]\times\mathcal{K} & &\rTo^{\widetilde{\textsf{ev}}} & & C[k;\R^3], 
\end{diagram}
\end{equation}
where $\pi_k$ is the usual projection onto the first $k$ coordinates and 
\[
\widetilde{\textsf{ev}}(\,\cdot\,,K):C[k;S^1]\longrightarrow C[k;\R^3]
\]
 is the evaluation map induced from the knot embedding map $K:S^1\hookrightarrow \R^3$; see Proposition \ref{prop:conf-funct}.  In other words it is a ``lift'' of the  product map 
\begin{equation}\label{eq:ev-map}
\begin{split}
\textsf{ev}\colon  C(k;S^1)\times\mathcal{K} & \longrightarrow C(k;\R^3)  \\
 ((t_1,\cdots,t_k),K) & \longmapsto (K(t_1),\cdots, K(t_k)) 
\end{split}
\end{equation}
\no to the compactified spaces.
All maps in Diagram \eqref{diag:bt} are smooth maps of manifolds with corners \cite{Bott-Taubes94,Sinha04}, which is equivalent to saying that they admit smooth extensions to 
some open neighborhoods of the domains of their charts.

Returning now to the diagram algebra $\mathcal{D}$, for  a trivalent diagram $D\in \mathcal{D}_n$, define the associated Gauss map to be the product
\begin{equation}\label{eq:h_D}
 h_D=\prod_{(i,j)\in \mathcal E(D)} h_{i,j}\colon C[k,s;\mathcal{K},\R^3]\longrightarrow \prod_{(i,j)\in \mathcal{E}(D)} S^2,
 %\qquad h_D=\prod_{(i,j)\in E(D)} h_{i,j},
\end{equation}
where 
$h_{i,j}\colon C[k,s;\mathcal{K},\R^3]\longrightarrow S^2$ is the lift to the compactification of the classical Gauss map  
\begin{align*}
C(k+s;\R^3) & \longrightarrow S^2,\\ 
(x_1,..., x_i,..., x_j,...,x_{k+s}) & \longmapsto \frac{x_j-x_i}{|x_j-x_i|}.
\end{align*}
  Maps $h_{i,j}$ extend smoothly to the boundary of $C[k,s;\mathcal{K},\R^3]$,  \cite[Appendix]{Bott-Taubes94}.  Thus $h_D$ is also smooth, and as a result we obtain
a smooth $(k+3s)$--form $\omega_D$ on $C[k,s;\mathcal{K},\R^3]$ via the pullback:
\begin{equation}\label{eq:omega_D}
 \omega_D=h^\ast_D(\omega\times\cdots\times\omega)=\prod_{(i,j)\in \mathcal{E}(D)} \omega_{i,j},\qquad \omega_{i,j}=h^\ast_{i,j}\omega.
\end{equation}
Here $\omega$ is the area form on $S^2$, usually chosen in standard coordinates on $\R^3$ as 
\[
 \omega(x,y,z)=\frac{x\,dy\wedge dz-y\,dx\wedge dz+z\,dx\wedge dy}{(x^2+y^2+z^2)^{\frac{3}{2}}}.
\]
One now has a smooth bundle of manifolds with corners,
\[
p_\mathcal{K}:C[k,s;\mathcal{K},\R^3]\longrightarrow \mathcal{K},
\]
which is the composition of $\bar{\pi}_k$ with the trivial projection of $C[k;S^1]\times\mathcal{K}$ onto the second factor.  The fiber of 
$p_\mathcal{K}$ over a knot $K$ is the configuration space of $k+s$ points in $\R^3$, first $k$ of which are constrained to lie on $K$.
Integration along the $(k+3s)$ dimensional fiber of $p_\mathcal{K}$
produces a 0-form (a function) on $\mathcal{K}$. We will denote its value at $K\in\mathcal{K}$  by $I_D(K)$.  In other words,
\begin{equation}\label{eq:I_D(K)}
 I_D(K):=\bigl((p_{\mathcal{K}})_\ast \omega_D\bigr)(K).
\end{equation}
\begin{remark}\label{rem:w_D-integrable}
 {\rm
  Note that $\omega_D$ vanishes to the order $1/r^{n}$ at ``infinity'' of $C[k+s;\R^3]$, where $r$ is the distance from the origin.  It is therefore integrable along fibers of $p_{\mathcal{K}}$ and thus $(p_{\mathcal{K}})_\ast \omega_D$ is well-defined.
 }
\end{remark}
 We now have the following fundamental result originally due to Altschuler and Freidel \cite{Altschuler-Freidel95}, but reproved by Thurston \cite{Thurston99} in the form we use here. 
\begin{theorem}[\cite{Altschuler-Freidel95, Thurston99}]\label{thm:vassiliev}
Given a primitive weight system $W\in \mathcal{W}_n$, $n\geq 0$, the map defined by
\begin{align}
  V_W:\mathcal{K} & \longrightarrow \R \label{eq:T_W}\\
  K& \longmapsto \frac{1}{(2n)!} \sum_{D\in TD_n}W(D)\bigl(I_D(K)-m_D\,I_{\includegraphics[scale=.06]{writhe2}}(K)\bigr),\notag
\end{align}
where $m_D$ is a real number which depends only on $D$, is a finite type $n$ knot invariant. Moreover, any finite type invariant of type $n$ can be expressed as $V_W$ for some primitive weight system $W\in \mathcal{W}_n$.  More precisely, $V_W$ gives an isomorphism $\mathcal V_n/\mathcal V_{n-1}\cong \mathcal W_n$ for all $n\geq 0$ (where by $\mathcal V_{-1}$ we mean the one-dimensional space of constant invariants).
\end{theorem}
\no Notice that the statement above  is a more elaborate version of \eqref{eq:V_W-sum}, with $a=\frac{1}{2n!}W(D)$ and $b=\frac{1}{2n!}m_D$.
The term $m_D\,I_{\includegraphics[scale=.06]{writhe2}}(K)$ is known as the {\em anomalous correction}.  The integral $I_{\includegraphics[scale=.06]{writhe2}}(K)$ computes the {\em writhing number}\footnote{Note that the writhing number is an average {\em writhe} over all possible projections of the knot.} of $K$ (see \cite{Calugareanu59}). 

 We next wish to 
clarify some technical aspects of the integration in \eqref{eq:I_D(K)} that will be needed for later constructions. Let $K\in \mathcal{K}$ and $D\in \mathcal{D}_n$.  The Gauss map $h_D$ 
from \eqref{eq:h_D} factors as $h_D=\bar{h}_D\circ p_{k,s}$ (see Diagram \eqref{diag:bt}), where 
\[
\bar{h}_D:C[k+s,\R^3]\longrightarrow \prod_{(i,j)\in \mathcal E(D)} S^2
\]
 has an identical definition as $h_D$ in \eqref{eq:h_D}. We also have the analog of the form $\omega_D$ on 
$C[k+s,\R^3]$, given as $\bar{h}_D^*(\omega\times\cdots\times \omega)$. We will also denote this form by $\omega_D$. Integrating along the $3s$-dimensional fiber of $\pi_k$ in Diagram \eqref{diag:bt},
we obtain a smooth $k$-form (see \cite[p. 5281]{Bott-Taubes94})
\begin{equation}\label{eq:varpi-D}
 \varpi_D=(\pi_k)_\ast \omega_D
\end{equation}
on $C[k;\R^3]$ (see Remark \ref{rem:w_D-integrable}).
On the other hand, the evaluation map \eqref{eq:ev-map}, produces 
at each point 
\[
\textsf{ev}((t_1,\cdots,t_k),K)=(K(t_1),\cdots,K(t_k))\in C(k;\R^3),\qquad (t_1,\cdots,t_k)\in C(k;S^1),
\] 
a frame
\[
 \dot{K}_k=\{\dot{K}(t_1),\cdots,\dot{K}(t_k)\},\qquad \dot{K}(t_i)=\frac{d}{dt_i} K(t_i).
\]
Lifting this frame to $C[k;\R^3]$, via the map pushforward $\alpha_\ast$ induced by $\alpha=\alpha_{\R^3}$ from \eqref{eq:alpha_M}, we obtain the frame $\widetilde{\dot{K}}_k=\alpha_\ast\dot{K}_k$.  Contracting  into $\varpi_D$, given by \eqref{eq:varpi-D}, we obtain a (distributional) function over $C[k;\R^3]$, 
determined by  
\begin{equation}\label{eq:f_D,K}
  f_{D,K}(\mathbf{t}):=\varpi_D(\widetilde{\dot{K}}_k)(\mathbf{t})=\alpha^\ast\varpi_D(\dot{K}_k)(\mathbf{t})=(\textsf{ev}^\ast_K\alpha^\ast\varpi_D)(\mathbf{t})[\partial_{\mathbf{t}}],
\end{equation}
for $\mathbf{t}=(t_1,\cdots,t_k)$.  Here  $\alpha^\ast$ denotes the pullback induced by $\alpha$. 
\begin{proposition}\label{prop:bt-integral-f_D,K}
 With $f_{D,K}$ as defined in \eqref{eq:f_D,K}, we have the following identity for $I_D(K)$:
\begin{equation}\label{eq:bt-integral-f_D,K}
I_D(K)=\underbracket[0.5pt]{\int^{T}_0\cdots\int^{T}_0}_{k\text{ times}} f_{D,K}(\mathbf{t})\, d\mathbf{t},
\end{equation}
where the interval $[0,T]$ parametrizes the knot $K$.
\end{proposition}
\begin{proof}
 Restricting $\bar{\pi}_k$ in Diagram \eqref{diag:bt} to the fiber over the point $K\in \mathcal{K}$ and the rest of the maps to the subset of the interior of the compactifications $C(\,\cdot\,;\R^3)\subset C[\,\cdot\,; \R^3]$, we have a diagram
%%%%%%%%%%%%%%%%%%%%%%%%%%%%%%%%
\begin{diagram}
	C_0(k,s;K,\R^3) & & \rInto^{j} & & C_0(k;S^1) \times C_0(k+s;\R^3) & &\rTo^{v} & & C_0(k+s;\R^3) \\
	 & \rdTo(4,2)^{\bar{\pi}_k}	&  & & \dTo & &  & &   \dTo_{\pi_k} \\
  & & & & C_0(k;S^1) & &\rTo^{\textsf{ev}_K} & & C_0(k;\R^3),
\end{diagram}
where $p=v\circ j$. Since $C(k+s;\R^3)\subset (\R^3)^k\times (\R^3)^s$, we may work with the standard coordinates $(\mathbf{x},\mathbf{y})=(x_1,\cdots,x_k,y_1,\cdots, y_s)$ on $(\R^3)^k\times (\R^3)^s$ imposed by the blowup map $\alpha:C(k+s;\R^3)\longrightarrow C_0(k+s;\R^3)\subset C[k+s;\R^3]$ of \eqref{eq:alpha_M}. Let components of vectors $x_j$ and $y_j$ in $(\mathbf{x},\mathbf{y})$ be further indexed as $x_j=(x^i_j)_{i=1,2,3}$, $y_j=(y^i_j)_{i=1,2,3}$ and denote by $\mathbf{t}=(t_1,\cdots,t_k)$ the coordinates on $C(k;S^1)\subset (S^1)^k$.
Then the $(k+3s)$--form $\alpha^\ast\omega_D$ can be written as 
\[
 \alpha^\ast\omega_D=d\mathbf{y}\wedge\hat{\alpha}_D +\doublehat{\alpha}_D,
\]
where $d\mathbf{y}$ is the top degree form on $(\R^3)^s$ and $\doublehat{\alpha}_D$ some $(k+3s)$--form not containing the term $d\mathbf{y}$. Using the multi-index notation, let us write
\[
\hat{\alpha}_D=\sum_{I,J} \hat{a}_{I,J}(\mathbf{x},\mathbf{y}) d\mathbf{x}^J_I.
\]  
Here $d\mathbf{x}^J_I=dx^{i_1}_{j_1}\wedge\cdots\wedge dx^{i_k}_{j_k}$ and $I$, $J$ are appropriate multiindices. After integrating along the fiber, we get
\[
 (\pi_k)_\ast\alpha^\ast\omega_D=\alpha^\ast\varpi_D(\mathbf{x})=\sum_{I,J} \Bigr(\int_{\pi^{-1}_k({\bf x})} \hat{a}_{I,J}(\mathbf{x},\mathbf{y})\, d\mathbf{y}\Bigl) d\mathbf{x}^J_I,
\] 
where $\pi^{-1}_k({\bf x})=C(s,\R^3-\{x_1,\cdots,x_k\})$.\footnote{I.e.~$\pi^{-1}_k({\bf x})$  is a configuration space of $s$ points in $\R^3$ with $k$ points deleted, see \cite{Fadell-Husseini01}.} From 
\eqref{eq:f_D,K}, we have
\[
 f_{D,K}(\mathbf{t})  =\sum_{I,J} \Bigr(\int_{\pi^{-1}_k(K(\mathbf{t}))} \hat{a}_{I,J}(K(t_1),\cdots,K(t_k),\mathbf{y})\, d\mathbf{y}\Bigl) b^J_I(t_1,\cdots,t_k), 
\]
where $b^J_I(t_1,\cdots,t_k)=d\mathbf{x}^J_I[\dot{K}^{\wedge_k}]$ and
\[
(\textsf{ev}^\ast_K\alpha_\ast\varpi_D)(\mathbf{t})=f_{D,K}(\mathbf{t})\,d\mathbf{t},\qquad d\mathbf{t}=dt_1\wedge\cdots\wedge dt_k.
\]
On the other hand, $v^\ast\alpha^\ast\omega_D$ has the identical expression as $\alpha^\ast\omega_D$, so we may write
$j^\ast\alpha^\ast\omega_D$ for $j^\ast v^\ast\alpha^\ast\omega_D$. In the $(\mathbf{t},\mathbf{x},\mathbf{y})$ coordinates, we obtain
\[
 j^\ast\alpha^\ast\omega_D=\sum_{I,J} \hat{a}_{I,J}(K(t_1),\cdots,K(t_k),\mathbf{y})\,b^J_I(t_1,\cdots,t_k)\,d\mathbf{y}\wedge dt_1\wedge\cdots\wedge dt_k.
\]
Since the boundary of the configuration space $C[k;S^1]$ is measure zero, it 
does not contribute to the integral in \eqref{eq:I_D(K)} and we easily see that the following identities, proving \eqref{eq:bt-integral-f_D,K}, hold:
\[
\begin{split}
I_D(K) & =\bigl((p_{\mathcal{K}})_\ast \alpha^\ast\omega_D\bigr)(K)=\int_{C(k;S^1)}(\bar{\pi}_k)_\ast(\alpha^\ast\omega_D)\\
 & =\int_{C(k;S^1)}(\bar{\pi}_k)_\ast(j^\ast\alpha^\ast\omega_D)=\int_{C(k;S^1)} f_{D,K}(\mathbf{t})\,d\mathbf{t}.%\qedhere
\end{split}
\]
\end{proof}

%%%%%%%%%%%%%%%%%%%%%%%%%%%%%%%%%%%%%%%%%%%%%%%%%%%%%%%%%%%%%
\section{Proofs of Theorems \ref{thm:main-knots} and \ref{thm:main-ergodic}}\label{sec:proofs}
%\section{Proofs of the main results}

In the setting of a volume-preserving vector field $X$ on a domain $\mathcal{S}$ from Theorem \ref{thm:main-knots}, we wish to apply constructions of Section \ref{sec:conf-integrals} to the family of knots $\{\bar{\mathscr{O}}_T(x)\}$ obtained from the ``closed up'' orbits of $X$. Note that any such orbit $\bar{\mathscr{O}}_T(x)$ (as in \eqref{eq:O-orbits}) is generically a piecewise smooth knot in $\R^3$. In order to define $\bar{\mathscr{O}}_T(x)$, one needs a system $\{\sigma(x,y)\}$ of {\em short paths} on $\mathcal{S}$, which can in general be defined from geodesics  after an appropriate choice of the metric on $\mathcal{S}$ \cite{Vogel03}. Short paths will in particular be dealt with in Lemma \ref{lem:short-paths}.  The main property of short paths we will use is that their length is uniformly bounded.  Note that we can assume $\bar{\mathscr{O}}_T(x)$ is smooth, because its corners can be rounded and $\bar{\mathscr{O}}_T(x)$ is the $C^0$ limit of these ``rounded'' parametrizations.

Recall that the basic ingredient of the formula \eqref{eq:T_W} for any finite type $n$ invariant $V_W$ is the integration function $I_D$ associated with a diagram $D\in \mathcal{D}_n$.  Following the ideas outlined in the Introduction we focus on the family of functions 
\begin{align*}
\mathcal{S} & \longrightarrow \R, \\
x  & \longmapsto I_D(\bar{\mathscr{O}}_T(x)) 
\end{align*}
that is dependent on $T$. For any $x\in\mathcal{S}$, we wish to study the time average
\begin{equation}\label{eq:bar-lambda_D}
 \bar{\lambda}_D(x)=\lim_{T\to \infty} \frac{1}{T^k} I_D(\bar{\mathscr{O}}_T(x)),\qquad k=k(D).
\end{equation}
Naturally, we need to investigate if $\bar{\lambda}_D$ is a well-defined function on $\mathcal{S}$ and whether it is integrable.

Recall that $X(x)=\dot{\mathscr{O}}_T(x)$.  Given a smooth $k$--form $\varpi_D$ on $C[k;\R^3]$ as defined in \eqref{eq:varpi-D}, we have a global analog of the function $f_{D,K}$  in \eqref{eq:f_D,K}:
\begin{equation}\label{eq:f_D,X}
\begin{split}
 f_{D,X}: C(k;\mathcal{S})\longrightarrow \R,\qquad f_{D,X}:=\alpha^\ast\varpi_D(X,\cdots,X),
 \end{split}
\end{equation}
where the frame of fields $\{X,\ldots,X\}$ spans the tangent space to the product of orbits $\mathscr{O}(x_1)\times \cdots\times\mathscr{O}(x_k)$ through any point $(x_1,\cdots,x_k)$ in $\mathcal{S}^k$. 
It is convenient to think about the above constructions in terms of the underlying foliation $\mathscr{F}^k_X$ of $\mathcal{S}^k$ defined by the orbits of the action of the $k$--fold product flow $\phi_{X}^{k}$ on $\mathcal{S}^k$. Note that $\mathscr{F}^k_X$ has complete leaves because $X$ is tangent to $\partial \mathcal{S}$, and orbits $\mathscr{O}(x)$ thus exist for all time. The function 
$f_{D,X}$ is well-defined on $C(k;\mathcal{S})$, but, except for along the orbits, it generally blows up close to the diagonals of $\mathcal{S}^k$. We can also consider the function 
\begin{equation}\label{eq:tilde-f_D,X}
 \begin{split}
 \widetilde{f}_{D,X}: C_0(k;\mathcal{S})\longrightarrow \R,\qquad f_{D,X}:=\varpi_D(\widetilde{X},\cdots,\widetilde{X}),
 \end{split}
\end{equation}
where $\{\widetilde{X},\ldots,\widetilde{X}\}$ is a lift of the frame $\{X,\ldots,X\}$ of vector fields to $C[k;S]$. Note that even though $\varpi_D$ is smooth on $C[k;\mathcal{S}]$, the vector field $\widetilde{X}=\alpha_\ast X$ undergoes ``infinite stretching'' close to the boundary of $C[k;\mathcal{S}]$ (see Remark \ref{rem:lemma-illustration}). Clearly, $f_{D,X}$ factors as 
\[
 f_{D,X}=\widetilde{f}_{D,X}\circ\alpha.
\] 

Since $\bar{\mathscr{O}}_T(x)$ is parametrized by the interval $[0,T+1]$ (where $[T,T+1]$ parametrizes a 
short segment $\sigma(x,\phi_X(x,T))$), Proposition \ref{prop:bt-integral-f_D,K} applied to 
\eqref{eq:bar-lambda_D} pointwise yields
\begin{equation}\label{eq:bar-lambda_D2}
 \begin{split}
 I_D(\bar{\mathscr{O}}_T(x)) & =  \underbracket[0.5pt]{\int^{T+1}_0\cdots\int^{T+1}_0}_{k\text{ times}} f_{D,X}
((\phi\cup\sigma)(x,t_1),\cdots, (\phi\cup\sigma)(x,t_k))\,dt_1\cdots dt_k,\\
 \bar{\lambda}_D(x) & =\lim_{T\to \infty} \frac{1}{T^k}  I_D(\bar{\mathscr{O}}_T(x)). %\quad x\in \mathcal S,
\end{split}
\end{equation}
Here $(\phi\cup\sigma)$ is a shorthand notation for the flow $\phi$ of $X$ followed by the short path parametrization.

This is thus the setup in which we will prove our main theorems in this section.  But before we can do that, we will establish a useful lemma.

%%%%%%%%%%%%%%%%%%%%%%%%%%%%%%%%%%%%%%%%%%%%%%%%%%%%%%%%%%%%%
\subsection{Key Lemma}

Here is the lemma that will be used in the proofs of Lemma \ref{lem:short-paths} and Theorems \ref{thm:main-knots} and \ref{thm:main-ergodic}.

%%%%%%%%%%%%%%%%%%%%%%%
\begin{klemma*}%\label{lem:lambda_D-integrable}
Let  $\mu$ be the underlying measure on the domain $\mathcal S\subset \R^3$, invariant under the flow of $X$.  Consider the time average of $f_{D,X}$ over $\mathscr{F}^k_X$, defined as
\begin{equation}\label{eq:lambda_D}
 \lambda_D(x)=\lim_{T\to \infty} \frac{1}{T^k} \int^{T}_0\cdots\int^{T}_0 f_{D,X}
(\phi(x,t_1),\cdots, \phi(x,t_k))\,dt_1\cdots dt_k,\quad x\in \mathcal S,
\end{equation}
where in comparison to \eqref{eq:bar-lambda_D}, we skipped the integrals over short paths.  Then this limit exists almost everywhere on $\mathcal S$ and $\lambda_D$ is in $L^1(\mathcal S,\mu)$.
\end{klemma*}
Before we prove this, we need to make several observations. Note that $\mu$ induces a measure on $\mathcal{S}^k$ by the pushforward via the thin diagonal inclusion $\mathcal S\hookrightarrow \mathcal S^k$, $x\longrightarrow (x,\cdots,x)$.  Let us denote the resulting measure by $\mu_\Delta$. Clearly $\mu_\Delta$ is a finite Borel measure supported on the thin diagonal of $\mathcal{S}^k$. Averaging over the $\R^k$--action of $\phi^k=\phi^k_X$ we obtain a $\phi^k$--invariant measure 
\begin{equation}\label{eq:bar-mu-delta}
 \bar{\mu}_\Delta=\lim_{T\to \infty} \frac{1}{T^k} \int^{T}_0\cdots\int^{T}_0 \bigl((\phi^k)_\ast \mu_{\Delta}\bigr)\,dt_1\,\cdots\,dt_k.
\end{equation}
For the $k$--fold product $\mathcal{S}^k$, the above is a well defined limit in the space of Borel measures $\mathcal{M}(\mathcal{S}^k)$, c.f.~\cite{Billingsley99}. Note that $\bar{\mu}_\Delta$ is supported on the set of leaves of the foliation $\mathscr{F}^k_X$ intersecting the thin diagonal in $\mathcal{S}^k$. From the definitions, we may write $\int_{\mathcal S}\lambda_D d\mu$ as 
\begin{equation}\label{eq:lambda_D=f_D,X}
\begin{split}
 \int_{\mathcal S}\lambda_D \mu & =\int_{\mathcal{S}^k}\Bigl(\lim_{T\to \infty} \frac{1}{T^k} \int^{T}_0\cdots\int^{T}_0 f_{D,X}
(\phi(x_1,t_1),\cdots, \phi(x_k,t_k))\,dt_1\cdots dt_k \Bigr)d\mu_\Delta\\
 & =\int_{\mathcal{S}^k} f_{D,X}\,d\bar{\mu}_\Delta,
\end{split}
\end{equation}
where in the third identity we used \eqref{eq:bar-mu-delta}. Therefore  the question of whether $\lambda_D$ is in $L^1(\mathcal S,\mu)$ is equivalent to the question of whether $f_{D,X}$ of \eqref{eq:f_D,X} is in $L^1(\mathcal S^k,\bar{\mu}_\Delta)$. 
\begin{remark}
{\rm
 In place of $\bar{\mu}_\Delta$ one can consider any other invariant measure,  in particular we may restrict just to any measure supported on
 the $k$--product $\mathscr{O}(x)\times\cdots\times \mathscr{O}(x)$ of a single long orbit, or equivalently obtained by averaging, as in \eqref{eq:bar-mu-delta}, a Dirac delta measure of a point $(x,\cdots,x)\in \mathcal{S}^k$. It is well known (c.f.~\cite{Billingsley99, Candel-Conlon00}) that $\bar{\mu}_\Delta$ can be arbitrarily well approximated by finite sums of such Dirac delta averages.  (We will use this fact in Section \ref{sec:q-helicity}.)
}
\end{remark}

\no In order to investigate integrability of $f_{D,X}$, we employ the following natural generalization of \cite[Proposition 10.3.2]{Candel-Conlon00} to a product of flows.
\begin{proposition}[\cite{Candel-Conlon00}]\label{prop:holonomy-invariant}
 Any $\phi^k$--invariant measure $\mu$ on $\mathcal{S}^k$ corresponds to  a holonomy invariant measure of the foliation
$\mathscr{F}^k_X$.  
\end{proposition}
\no Let us choose a finite regular foliated atlas for $\mathscr{F}^k_X$ where a domain $V^{\pmb\alpha}$, $\pmb\alpha=(\alpha_1,\cdots,\alpha_k)$, of each each chart is  a product of regular flow boxes $\{V_\alpha\}$ of the vector field $X$ covering $\mathcal S$.  In other words, 
\begin{equation}\label{eq:V^alpha}
 V^{\pmb\alpha}=V_{\alpha_1}\times\cdots \times V_{\alpha_k},\qquad  V_{\alpha_i}=\mathcal{T}_{\alpha_i}\times I_{\alpha_i},\quad I_{\alpha_i}=(-\epsilon_{\alpha_i},\epsilon_{\alpha_i}),\quad 0<\epsilon_{\alpha_i}<\epsilon,
\end{equation}
where each $\mathcal{T}_{\alpha_i}$ is a transverse disk to the flow of $X$. $V^{\pmb\alpha}$ can be expressed as the product 
\[
 V^{\pmb\alpha}=\mathcal{T}^{\pmb\alpha}\times I^{\pmb\alpha}=\bigl(\prod_{i} \mathcal{T}_{\alpha_i}\bigr)\times \bigl(\prod_{i} I_{\alpha_i}\bigr).
\]
 Recall from \cite{Candel-Conlon00} that a holonomy invariant measure $\nu_{\mathscr{F}}$ of $\mathscr{F}=\mathscr{F}^k_X$ is a measure defined on $\bigsqcup_{\pmb\alpha} \mathcal{T}^{\pmb\alpha}$ that is invariant under the action of the holonomy pseudogroup of $\mathscr{F}$. The Rulle--Sullivan Theorem \cite{Ruelle-Sullivan75} (see also \cite[p. 245]{Candel-Conlon00}) and Proposition \ref{prop:holonomy-invariant} imply the existence of a  holonomy invariant (finite) measure  $\nu_{\mathscr{F}}$ corresponding to $\bar{\mu}_\Delta$, given in \eqref{eq:bar-mu-delta}, such that 
\begin{equation}\label{eq:lambda_D-current}
 \int_{\mathcal{S}^k} f_{D,X}\,\bar{\mu}_\Delta=\sum_{\alpha} \int_{\mathcal{T}^{\pmb\alpha}} \Bigl(\int_{I^{\pmb\alpha}} \xi_{\alpha}(\mathbf{x},\mathbf{t}) f_{D,X}(\mathbf{x},\mathbf{t}) d\mathbf{t}\Bigr) d\nu_{\mathscr{F}}(\mathbf{x}),
\end{equation}
where $(\mathbf{x},\mathbf{t})=(x_1,\cdots x_k,t_1,\cdots, t_k)$ are coordinates on $V^{\pmb\alpha}$, $\{\xi_{\pmb\alpha}\}$ a partition of unity subordinate to the cover of $\mathcal{S}^k$ by $\{V^{\pmb\alpha}\}$, and $d\mathbf{t}$ is induced from the usual Riemannian length measure along the orbits of $X$.

%%%%%%%%%%%%%%%%%%%%%%%%%%%%%%%%%%%%
\begin{remark}[Illustration for proof of Key Lemma]\label{rem:lemma-illustration}
{\rm 
Let us illustrate our strategy in the case of the simplest diagram $D=\includegraphics[scale=.07]{writhe2}$. For a fixed $\phi_X$--invariant measure $\mu$ on $\mathcal{S}$, the question is whether the following time average is $\mu$--integrable:
\[
 \lambda_D(x)=\lim_{T\to \infty} \frac{1}{T^2} I_D(\mathscr{O}_T(x)),\qquad x\in\mathcal{S}
\]
For simplicity, here we disregard short paths. 

%Using the identity \eqref{eq:lambda_D=f_D,X}, this question is equivalent to the question of whether $f_{D,X}=f_{\includegraphics[scale=.06]{writhe2},X}$ is integrable with respect to the $\phi^2_X$--invariant  diagonal measure $\bar{\mu}_\Delta$ defined in \eqref{eq:bar-mu-delta}. 

Considering a finite cover of $\mathcal{S}^2$ by flowboxes $V^{\pmb\alpha}$ (defined for $k=2$ in \eqref{eq:V^alpha}), formula \eqref{eq:lambda_D-current} tells us that it suffices to prove that 
$f_{D,X}$ is locally integrable with respect to $d\mathbf{t}\times\nu_{\mathscr{F}}$ in each $V^{\pmb\alpha}$. Away from the diagonal $\Delta_{\{1,2\}}$ of $\mathcal{S}^2$, $f_{D,X}$ is smooth, and so the hardest case is that of flowboxes intersecting $\Delta_{\{1,2\}}$. Without loss of generality consider a flowbox $V^{\pmb\alpha}$, $\pmb\alpha=(\alpha_1,\alpha_2)$ with $\alpha_1=\alpha_2$. To simplify the notation we denote it by $V=(\mathcal{T}\times I)\times (\mathcal{T}\times I)$, 
(where $\mathcal{T}\times I$ is a flowbox of $X$ in $\mathcal{S}$, with $I=(-\epsilon,\epsilon)$). Let $F:\mathcal{T}\times\mathcal{T}\to \R$ be defined by 
\begin{equation}\label{eq:F-int}
F(x,y)=\int_{I(x,y)} f_{D,X}(x,y,\mathbf{t}) d\mathbf{t},
\end{equation}
where 
\[
 I(x,y)=\{x\}\times\{y\}\times I\times I\subset V,\qquad I=(-\epsilon,\epsilon).
\] 
Since $C[2;\mathcal{S}]\subset C[2;\R^3]$ and $C[2;\R^3]$ is obtained by blowing up the diagonal $\Delta_{\{1,2\}}$ of $(\R^3)^2$, we can construct the metric $\widetilde{g}$ on $C[2;\R^3]$ from the standard Euclidean metric of $\R^3$ and pull it back to $C[2;\mathcal{T}\times I]$ via the map $\widetilde{\phi^2}$ induced from the product flow $\phi^2=\phi\times\phi$. The resulting metric on $C[2;\mathcal{T}\times I]$ will also be denoted by $\widetilde{g}$.
Recall that $\varpi_D=\varpi_{\includegraphics[scale=.06]{writhe2}}$ is a smooth $2$--form on $C[2;\mathcal{S}]\subset C[2;\R^3]$, defined via the Gauss map in \eqref{eq:h_D}.  Thanks to Proposition \ref{prop:conf-funct}, it pulls back to a smooth form on $C[2;\mathcal{T}\times I]$. The resulting pullback form will also be denoted by $\varpi_D$. In the case of configurations of two points, $C[2;\mathcal{T}\times I]$, the blowup map \eqref{eq:alpha_M} can be set equal  to the map defined in \eqref{eq:blow-up-map}, namely
\[
 \beta:C(2;\mathcal{T}\times I)\longrightarrow C[2;\mathcal{T}\times I]=\operatorname{Bl}(\mathcal{T}\times I, \Delta_{\{1,2\}})
\]
Equations \eqref{eq:f_D,X} and \eqref{eq:tilde-f_D,X} imply
\begin{equation}\label{eq:F}
 F(x,y)=\int_{I(x,y)} \alpha^\ast\varpi_D=\int_{\widetilde{I(x,y)}} \varpi_D,
\end{equation}
where  $\widetilde{I(x,y)}$ is the lift of $I(x,y)\subset V$ to  $C[2;\mathcal{T}\times I]$. Using the metric $\widetilde{g}$, Lemma \ref{lem:A-estimate} yields 
\begin{equation}\label{eq:|F|}
  |F(x,y)| \leq A_{\varpi_D,\widetilde{g}} \operatorname{vol}(\widetilde{I(x,y)}).
\end{equation}
\smallskip
\no {\em Claim:} Volumes of lifts $\widetilde{I(x,y)}$ in the metric $\widetilde{g}$ are uniformly bounded over $\mathcal{T}\times \mathcal{T}$. 
\smallskip

Given the claim, estimate \eqref{eq:|F|} implies that $F$ is pointwise bounded and thus $\nu_{\mathscr{F}}$--integrable (because $\nu_{\mathscr{F}}$ is a finite measure). Applying this argument to each flow box chart $\{V^{\pmb\alpha}\}$, we can conclude 
 that $f_{D,X}$ is $\bar\mu_{\Delta}$--integrable as required. 
\begin{remark} 
{\rm One can regard the above reasoning as an alternative to the proof of Lemma 2.4 in \cite[p. 1429]{Contreras-Iturriaga99}.} 
\end{remark}
\no {\em Justification of Claim:} The claim is intuitively clear, because the blowup map $\beta$ ``stretches'' $I(x,y)$ locally by adding a ``bump'' (which is illustrated on the right side of Figure \ref{fig:blow-up}).  To give a more precise argument, recall that $C[2;\mathcal{T}\times I]$ is a subspace of $C[2;\R^2\times \R]$ 
and $C[2;\R^2\times \R]$ is diffeomorphic to $(\R^2\times \R)\times S^2\times [0,\infty)$, (i.e.~it is the complement of a tubular neighborhood of the thin diagonal). The blowup map 
\[
\beta:\Bigl((\R^2\times \R)^2-\Delta_{\{1,2\}}\Bigr)\longrightarrow \Bigl(C[2;\R^2\times \R]-\beta^{-1}(\Delta_{\{1,2\}})\Bigr)\cong C(2;\R^2\times\R),
\]
can be written explicitly as 
\[ 
 \beta:((x,s),(y,t))\longmapsto \Bigl(\frac{x+y}{2},\frac{s+t}{2},\frac{(s-t,x-y)}{\sqrt{(s-t)^2+|x-y|^2}},\frac{1}{2} \sqrt{(s-t)^2+|x-y|^2}\Bigr).
\] 
Let $(x,y)\in \mathcal{T}\times \mathcal{T}$, $x\neq y$, and set $p=\frac{1}{2}(x+y)$, $q=\frac{1}{2}(x-y)$, conveniently changing variables to $u=\frac{1}{2}(s+t)$, $v=\frac{1}{2}(s-t)$, $s,t\in (-\epsilon,\epsilon)$.  We obtain for $(u,v)\in (-\epsilon,\epsilon)\times (-\epsilon,\epsilon)$
\begin{equation}\label{eq:beta-u,v}
\begin{split}
((x,u),(y,v)) & \longmapsto \Bigl(p,u;\frac{(v,q)}{\sqrt{v^2+|q|^2}}, \sqrt{v^2+|q|^2}\Bigr)=:(\psi_p(u);\psi_q(v)), 
\end{split}
\end{equation}
which, for a fixed $x$ an $y$, gives a $(u,v)$--parametrization of the lift $\widetilde{I(x,y)}$.  Here $\psi_p$, $\psi_q$ denotes the curves given by respectively first  and last two coordinates of the map \eqref{eq:beta-u,v}. The volume $\operatorname{vol}(\widetilde{I(x,y)})$ can now be estimated as
\[
\operatorname{vol}(\widetilde{I(x,y)})\leq c_{\widetilde{g}}\,\ell(\psi_p)(\ell(\psi_q)+2\pi\epsilon),
\]
where $\ell(\psi_p)$, $\ell(\psi_q)$ are lengths of the curves $\psi_p$ and $\psi_q$ in the metric $\widetilde{g}$ and $c_{\widetilde{g}}$ is a constant which depends only on $\widetilde{g}$. Lengths $\ell(\psi_p)$ and $\ell(\psi_q)$ are proportional to $\epsilon$ and thus $(x,y)\longmapsto\operatorname{vol}(\widetilde{I(x,y)})$ is uniformly bounded on $\mathcal{T}\times\mathcal{T}$ by a constant dependent only on the metric and the size of the flow box neighborhood.  }\qed
\end{remark}
%%%%%%%%%%%%%%%%%%%%%%%%%%%%%%%%
\begin{figure}[htbp]
	\centering
		\frame{\includegraphics[scale=.42]{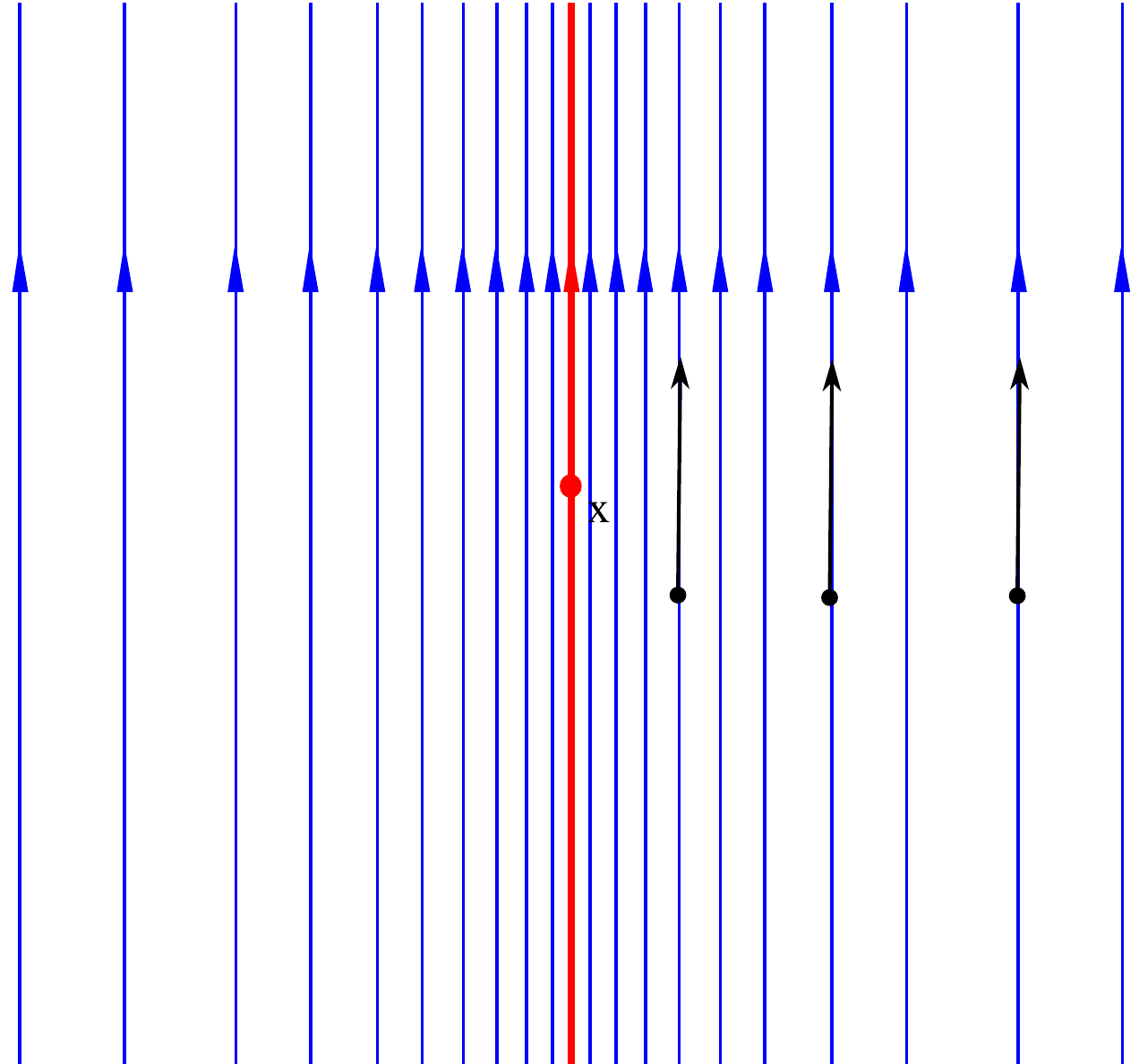}}\ \qquad
		\frame{\includegraphics[scale=.4]{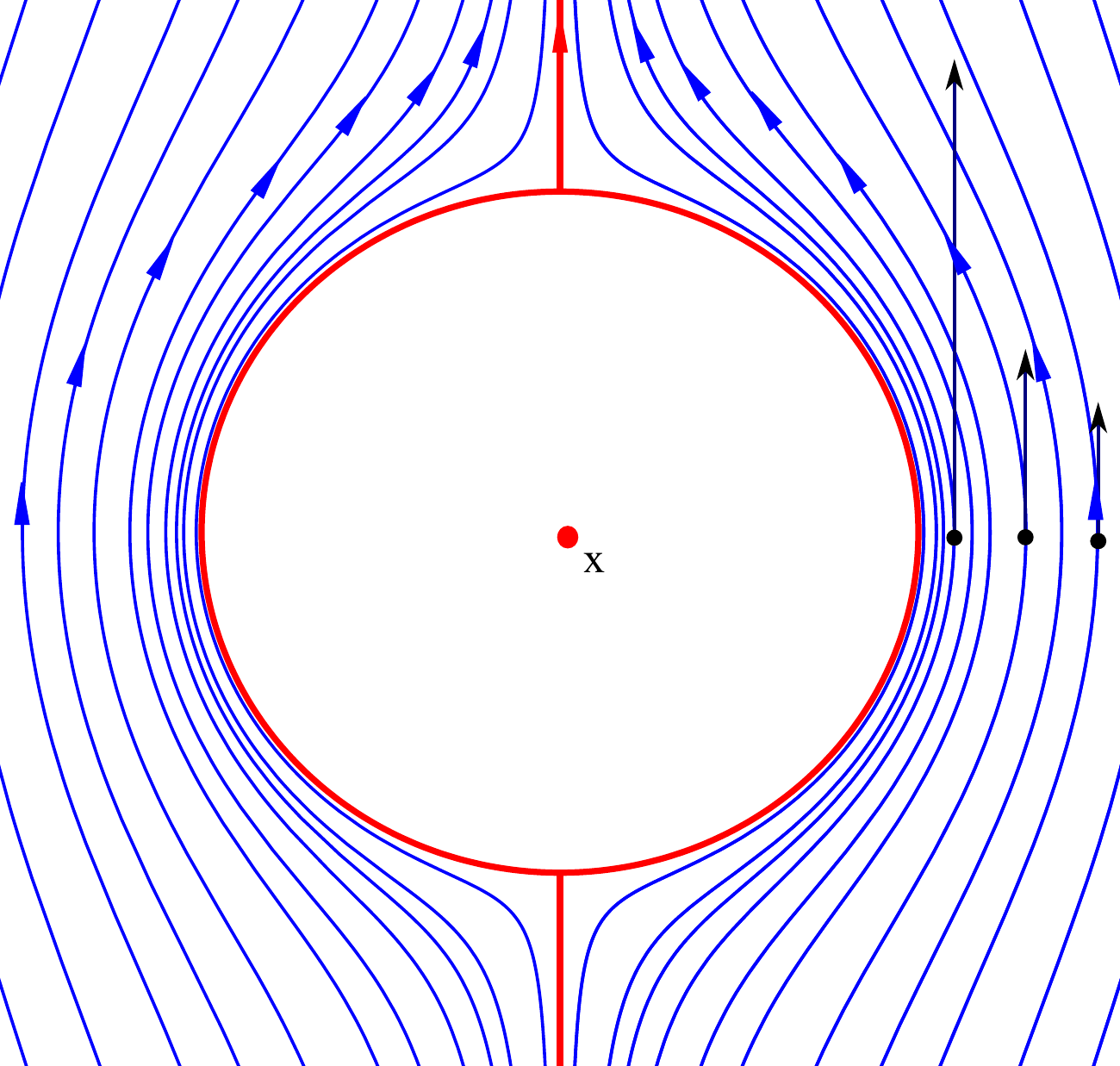}}
	\caption{{\small Lift of orbits of a vector field on $\R^2$  (left) to $\operatorname{Bl}(\R^2,\{x\})\cong S^1\times [1,+\infty)$ (right).}}\label{fig:blow-up} %Identifying $\operatorname{Bl}(\R^2,\{x\})$ with the complement of the unit disk in $\R^2$ with polar coordinates, the orbits are parametried by $r(\theta)=1+r_0/\cos\theta$ in for various values of $r_0$.}}	
\end{figure}
%\vspace{.2cm}
%%%%%%%%%%%%%%%%%%%%%%%%%%%%%%%%
\begin{proof}[Proof of Key Lemma]
Fix a flowbox chart $V^{\pmb\alpha}=\mathcal{T}^{\pmb\alpha}\times I^{\pmb\alpha}$ as defined in \eqref{eq:V^alpha}. It suffices to prove that the function 
\begin{align} 
F_{D,\pmb\alpha}\colon \mathcal{T}^{\pmb\alpha} & \longrightarrow \R \label{eq:F_D} \\
\mathbf{x} & \longmapsto \int_{I^{\pmb\alpha}}  f_{D,X}(\mathbf{x},\mathbf{t}) d\mathbf{t} \notag
\end{align}
is bounded for any $\pmb\alpha$. Then, because the atlas $\{V^{\pmb\alpha}\}$ is finite, \eqref{eq:lambda_D-current} implies $|\int_{\mathcal{S}^k} f_{D,X}\,\bar{\mu}_\Delta|\leq \infty$ as required.
Note that  $f_{D,X}$ is smooth away from the diagonals $\Delta_Q$ of $\mathcal{S}^k$, where $Q\subset \{1,\cdots, k\}$, $\#Q\geq 2$, and
\[
 \Delta_Q=\{(x_1,\cdots,x_k)\in \mathcal{S}^k\ |\ x_i=x_j\ \text{for}\ i,j\in Q\}.
\]
Let $\Delta=\bigcup_Q \Delta_Q$. We can cover the $\epsilon$--neighborhood of $\Delta$ by open sets 
\begin{equation}\label{eq:V(alpha,beta)}
\begin{split}
 V^{\pmb\alpha}_Q=\prod^k_{i=1}  V_{\alpha_i}, & \quad \text{where}\quad \alpha_i=\alpha_j,\quad\text{for}\quad i,j\in Q,\\
 &\quad\text{and}\quad \overline{V}_{\alpha_i}\cap \overline{V}_{\alpha_j}=\text{\O},\quad \text{for}\quad i\in Q,\ j\not\in Q.
\end{split}
\end{equation}
Then the sup--norm of $f_{D,X}$ is bounded on the complement of the $\varepsilon$--neighborhood of $\Delta$ by some constant which only depends on $X$ and $\varpi_D$; see \eqref{eq:f_D,X}. Generally, we want to pick $\varepsilon$ much smaller than $\epsilon$, which is the size of the flow box charts.
Thus, it suffices to prove that the functions $F=F_{D,\pmb\alpha}$ are bounded on $V^{\pmb\alpha}_Q$ for any $\pmb\alpha$ and $Q$.  Up to a permutation of factors, suppose that $Q=\{1,...,r\}\subset \{1,\cdots,k\}$, $2\leq r\leq k$, $r=\#Q$ and 
$\pmb\alpha=(\alpha_1,\cdots,\alpha_{r},\beta_1,\cdots,\beta_{k-r})$ (for $\alpha=\alpha_1=\cdots=\alpha_{r}$). Then $V^{\pmb\alpha}_Q=\mathcal{T}^r\times I^r\times \mathcal{T}^\beta\times I^\beta$, where
\[
\mathcal{T}=\mathcal{T}_{\alpha},\quad  I=I_{\alpha},\quad \mathcal{T}^\beta=\mathcal{T}_{\beta_1}\times\cdots\times \mathcal{T}_{\beta_{k-r}},\quad  \text{and}\quad  I^\beta=I_{\beta_1}\times\cdots\times I_{\beta_{k-r}}. 
\]
 Proposition \ref{prop:conf-funct} implies that the flow $\phi=\phi_X$ of $X$ restricted to the flow box $\mathcal{T}\times I$ lifts to  an embedding 
$\widetilde{\phi}:C[r;\mathcal{T}\times I]\longrightarrow C[r;\R^3]$, which, by the second condition 
in \eqref{eq:V(alpha,beta)}, extends trivially to the embedding 
\[
 \widetilde{\phi}_{Q}:W^{\pmb\alpha}_Q\longrightarrow C[k;\R^3],\qquad W^{\pmb\alpha}_Q:=C[r;\mathcal{T}\times I]\times \mathcal{T}^{\beta}\times I^{\beta}.
\]
Let $\overline{\alpha}_{Q}:W^{\pmb\alpha}_Q\longrightarrow V^{\pmb\alpha}_Q$ be the obvious projection induced by the restriction of the blowdown map (of Definition \ref{D:Compactification}). Recall that  
for any  $\mathbf{x}\in \mathcal{T}^r\times \mathcal{T}^\beta$, the lift $\widetilde{I_Q(\mathbf{x})}$ of $I_Q(\mathbf{x})=\{\mathbf{x}\}\times I^r\times I^\beta$ to $W^{\pmb\alpha}_Q$ equals to the closure of $\overline{\alpha}^{-1}_{Q}(I_{Q}(\mathbf{x}))$ in $W^{\pmb\alpha}_Q$.
The $k$--form $\varpi_D$ \eqref{eq:varpi-D} pulls back to a smooth form on $W^{\pmb\alpha}_Q$, and we may also pull back the metric $\widetilde{g}$ from $C[k;\R^3]$ to $W^{\pmb\alpha}_Q$. By \eqref{eq:f_D,X} and \eqref{eq:F_D}, a point value of $F$, for any $\mathbf{x}\in \mathcal{T}^r\times \mathcal{T}^\beta$, is given as
\begin{equation}\label{eq:F-int-w_D-lift}
 F(\mathbf{x})=\int_{\widetilde{I_Q(\mathbf{x})}} \varpi_{D}.
\end{equation}
Using Lemma \ref{lem:A-estimate}, for some universal constant $A_{D}$ we obtain a bound 
\begin{equation}\label{eq:F-vol-est}
|F(\mathbf{x})|\leq A_{D} \operatorname{vol}(\widetilde{I_Q(\mathbf{x})}).
\end{equation}
\no Therefore, analogously to what is outlined in  Remark \ref{rem:lemma-illustration}, it suffices to show that $\operatorname{vol}(\widetilde{I_Q(\mathbf{x})})$ is uniformly bounded over $\mathcal{T}^r\times \mathcal{T}^\beta$. This is intuitively clear because $\operatorname{vol}(I_Q(\mathbf{x}))$ is uniformly bounded by a constant proportional to $\epsilon^k$ (c.f.~\eqref{eq:V^alpha}) and $C[k;\R^3]$ is obtained from $(\R^3)^k$ by a sequence of blowups.  Hence the philosophy presented in Remark \ref{rem:lemma-illustration} implies that $\operatorname{vol}(\widetilde{I_Q(\mathbf{x})})$ is uniformly bounded as well. The remaining part of this proof provides details of this intuitive claim.

Summarizing, given a bounded ``flow box'': $W^{\pmb\alpha}_Q=C[r;\mathcal{T}\times I]\times \mathcal{T}^\beta\times I^\beta$ embedded via the flow in $C[k;\R^3]$, we intend to estimate the volume of the lift of $I_Q(\mathbf{x})=\{\mathbf{x}\}\times I^r\times I^\beta\subset (\mathcal{T}^r\times \mathcal{T}^\beta)\times I^r\times I^\beta=V^{\pmb\alpha}_Q$ to $W^{\pmb\alpha}_Q$ for every $\mathbf{x}\in \mathcal{T}^r\times \mathcal{T}^\beta$. Specifically, we consider  $V^{\pmb\alpha}_Q$ sitting in $\R^2\times \R$ where the $I$ factors of $V^{\pmb\alpha}_Q$ are mapped into the $\R$ factors under the inclusion, and the corresponding blowup map
%%%%%%%%%%%%%%%%%%%%%%%%%%%%%%%%%%%%%%%%%%%%%%%%
\begin{equation}\label{eq:compact-map}
\begin{split}
 \alpha:C(k;\R^2\times \R) & \longrightarrow A[k;\R^2\times \R],\\
 A[k;\R^2\times \R] & =(\R^2\times \R)^k\times \prod_{S\subset\{1, ..., k\},\ |S|\geq 2} \operatorname{Bl}((\R^2\times \R)^S,\Delta_S).
\end{split}
\end{equation}
Recall from Section \ref{sec:conf-integrals} that  $C[k;\R^2\times \R]$ is obtained as a closure of the graph of the above map. The projection of the map $\alpha$ to the first factor of $A[k;\R^2\times \R]$ is just the inclusion and the projections restricted to the $\operatorname{Bl}(\,\cdot\,)$ factors are determined by the blowup maps as in \eqref{eq:blow-up-map}. The metric on each $\operatorname{Bl}_S=\operatorname{Bl}((\R^2\times \R)^k,\Delta_S)$, further denoted by $\widetilde{g}_S(\varepsilon)$, is obtained via the construction of Section \ref{sec:blowups}. The parameter $\varepsilon$ is set to be sufficiently smaller then the diameter of the flowbox chart. Recall that  $\operatorname{Bl}_S$ is diffeomorphic to the complement of a tubular neighborhood of the thin diagonal $\Delta_S$ in $(\R^2\times \R)^S$, namely
%%%%%%%%%%%%%%%%%%%%%%%%%%%%%%%%%%%%%%%%%%%%%%%%
\begin{equation}\label{eq:bl-diff}
 \operatorname{Bl}_S\cong(\R^2\times \R)\times \mathbb{S}^{3 |S|-4}\times [0,\infty).
\end{equation}
and the map $\alpha$ restricted to factors of $A[k;\R^2\times\R]$, indexed by $S=\{s_1,\cdots,s_{|S|}\}$,  can be specifically chosen as
\begin{equation}\label{eq:FM-restriction}
 \mathbf{y}=(y_1,\cdots,y_{|S|}) \longmapsto \Bigr(y_1;\frac{y_1-y_2}{|\mathbf{y}'|},\cdots,\frac{y_1-y_{|S|}}{|\mathbf{y}'|};|\mathbf{y}'|\Bigl),\quad \mathbf{y}'=(y_1-y_2,\cdots, y_1-y_{|S|}).
\end{equation}
This gives an embedding into the interior of $\operatorname{Bl}_S$, i.e.~into $(\R^2\times \R)\times \mathbb{S}^{3 |S|-4}\times (0,\infty)$.

%Note that the flow $\phi_X$ induces a map $A[\phi_X]:A[k;\mathcal{T}\times I]\longrightarrow A[k;\mathcal{S}]$, where $A[k;\mathcal{T}\times I]\subset A[k;\R^2\times \R]$ and $A[k;\mathcal{S}]\subset A[k;\R^3]$, and restriction of $A[\phi_X]$ to $C[k;\mathcal{T}\times I]$, further denoted by $C[\phi_X]:C[k;\mathcal{T}\times I]\longrightarrow C[k;\mathcal{S}]$ gives a smooth map of manifolds with corners. Then $W^{\pmb\alpha}_Q$ is just a closure of the image of $C(r;\mathcal{T})\times \mathcal{T}^\beta\times I^\beta$ under the composition of $\alpha$ and $C[\phi_X]$.

For simplicity, suppose $\mathbf{x}\in \mathcal{T}^k\cong \mathcal{T}^r\times\mathcal{T}^\beta$ and $\mathbf{x}\in C(k;\mathcal{T})$, i.e.~$\mathbf{x}$ is away from the thick diagonal of $\mathcal{T}^k$. The restriction of the map $\alpha$ to $I(\mathbf{x})=\{\mathbf{x}\}\times I^k$,\footnote{Where we abbreviate $I^r\times I^\beta$ to $I^k$ with $k=r+|\beta|$.}  gives a parametrization of the lift $\widetilde{I(\mathbf{x})}$ in $C[k;\R^2\times\R]$. Let us denote this parametrization by $\gamma(t_1,\cdots,t_k)$, where $\mathbf{t}=(t_1,\cdots,t_k)$ are the variables of $I^k$ and $\widetilde{I(\mathbf{x})}=\gamma(I(\mathbf{x}))$.  Further let
\[
 X_i(\mathbf{t})=\frac{\partial}{\partial t_i} \gamma(\mathbf{t}),\qquad i=1,\cdots,k,
\]
be images  in $A[k;\R^2\times \R]$ of coordinate vector fields under the derivative $D \gamma$. Then, for $\widetilde{I_Q(\mathbf{x})}=\widetilde{\phi}\circ \alpha(I(\mathbf{x}))$ we have
\begin{equation}\label{eq:vol-lift}
 \operatorname{vol}(\widetilde{I_Q(\mathbf{x})})\leq c_{\phi}\operatorname{vol}(\widetilde{I(\mathbf{x})})=c_{\phi}\int_{I^k} (|X_1\wedge\cdots\wedge X_k|_{\widetilde{g}})^{\frac{1}{2}}\,d\mathbf{t},
\end{equation}
where $c_{\phi}$ accounts for the $C^1$--norm of the map $\widetilde{\phi}$. Each vector $X_i$ has coordinates $(X^j_i,X^S_i)$, where $j=1,\cdots,k$ indexes  factors: $(\R^2\times\R)^k$ in $A[k;\R^2\times \R]$ and $S\subset\{1, ..., k\}$, $|S|\geq 2$, indexes the $\operatorname{Bl}_S$ factors, we call $j$ the {\em front index} and $S$ the {\em set index} in the above decomposition of $X_i$. Substituting $X_i=\sum_{\pmb{m}} X^{\pmb{m}}_i$, where $\pmb m$ ranges over both $j$ and $S$ type indices, we estimate
\begin{equation}\label{eq:volume-wedge-expansion}
\begin{split}
 |X_1\wedge\cdots\wedge X_k|_{\widetilde{g}} & \leq \sum_{\pmb{m}=(\pmb{m}_1,\pmb{m}_2,\cdots,\pmb{m}_k)} |X^{\pmb{m}_1}_1\wedge\cdots\wedge X^{\pmb{m}_k}_k|_{\widetilde{g}}\\
 & \leq \sum_{\pmb{m}} \prod^k_{l=1} |X^{\pmb{m}_l}_l|_{\widetilde{g}} \leq \sum_{\pmb{m}} \frac{1}{k} \sum^k_{l=1} \bigl(|X^{\pmb{m}_l}_l|_{\widetilde{g}}\bigr)^k,
\end{split}
\end{equation}
where the last step is a consequence of the arithmetic mean and Jensen's inequality (c.f.~\cite{Lieb-Loss-book}). Consequently, estimating the integral
in \eqref{eq:vol-lift}, boils down to estimating integrals in the form 
\[
\mathcal{I}(\pmb{m}_l)=\int_{I^k} \bigl(|X^{\pmb{m}_l}_l|_{\widetilde{g}}\bigr)^k\, d\mathbf{t},\qquad l=1,\cdots, k.
\]
Without loss of generality (as we may always change the order of integration) suppose $l=1$, and let
\[
\mathcal{I}=\int_I\cdots\Bigl(\int_I \bigl(|X^{\pmb{m}_1}_1(t_1,\cdots,t_k)|_{\widetilde{g}}\bigr)^k\, dt_1\Bigr)\cdots dt_k.
\]
For a fixed $\mathbf{t}_0=(t_2,\cdots,t_k)$, the inner integral:
\[
 E_k(\gamma_{\pmb{m}_1})=\int_I \bigl(|X^{\pmb{m}_1}_1(t_1,\mathbf{t}_0)|_{\widetilde{g}}\bigr)^k\, dt_1,
\]
represents the $L^k$--energy\footnote{i.e.~the $L^k$--norm to the $k$th power.} of the curve parametrized by $\gamma_{\pmb{m}_1}:t_1\longrightarrow \pi_{\pmb{m}_1}(\gamma(t_1,\mathbf{t}_0))$, where $\pi_{\pmb{m}_1}$ is a projection onto the $\pmb{m}_1$--coordinate of $A[k;\R^2\times\R]$. If $\pmb{m}_1$ is a front index then $|X^{\pmb{m}_1}_1(t_1,\mathbf{t}_0)|_{\widetilde{g}}=1$ and $E_k(\gamma_{\pmb{m}_1})\leq \ell(\gamma_{\pmb{m}_1})$. Since $\ell(\gamma_{\pmb{m}_1})\leq c_\epsilon \epsilon$, for some constant $c_\epsilon> 0$, we obtain 
\[
\mathcal{I}(\pmb{m}_1)\leq (c_\epsilon \epsilon)^k.
\]
 In the case $\pmb{m}_1=S$ is a set index, the map  $\gamma_{\pmb{m}_1}$ parametrizes 
a curve in $\operatorname{Bl}_S$, which projected via  \eqref{eq:FM-restriction} onto the 
$\mathbb{S}^{3 |S|-4}$ factor is a ''piece'' of a great circle. Then a simple computation in the metric $\widetilde{g}_S(\varepsilon)$ leads to the following estimate
\[
 E_k(\gamma_{\pmb{m}_1})\leq c_{k,\epsilon} (2\pi+1)^k\epsilon^k,
\]
where we used $\varepsilon\ll\epsilon$, again $\mathcal{I}(\pmb{m}_1)$ is uniformly bounded. Applying $\int_{I^k}\,(\cdot)$ to both sides of  \eqref{eq:volume-wedge-expansion} we obtain from \eqref{eq:vol-lift} that $\operatorname{vol}(\widetilde{I_Q(\mathbf{x})})$ is estimated by a sum of $\mathcal{I}$--type terms. Therefore using estimates for $\mathcal{I}(\pmb{m}_1)$ we obtain the required uniform bound 
\[
 \operatorname{vol}(\widetilde{I_Q(\mathbf{x})})\leq  c_{k,\epsilon,\phi} (1+2\pi\epsilon)^k (\epsilon)^{2k}.
\]

In the case $\mathbf{x}$ belongs to the thick diagonal of $\mathcal{T}^k$, we obtain the above bound by considering $\mathbf{x}$ as a limit of points from $C(k;\mathcal{T})$.  
\end{proof}

%%%%%%%%%%%%%%%%%%%%%%%%%%%%%%%%%%%%%%%%%%%%%%%%%%%%%%%%%%%%%
\subsection{Short paths}

 We are now ready to show that the short paths do not contribute to the limits in \eqref{eq:bar-lambda_D} and \eqref{eq:bar-lambda_D2}.  
%%%%%%%%%%%%%%%%%%%%%%
\begin{lemma}\label{lem:short-paths}
% Consider the time average of $f_{D,X}$ over $\mathscr{F}^k_X$, defined as
%\begin{equation}\label{eq:lambda_D}
% \lambda_D(x)=\lim_{T\to \infty} \frac{1}{T^k} \int^{T}_0\cdots\int^{T}_0 f_{D,X}
%(\phi(x,t_1),\cdots, \phi(x,t_k))\,dt_1\cdots dt_k,\quad x\in \mathcal S,
%\end{equation}
%where in comparison to \eqref{eq:bar-lambda_D}, we skipped the integrals over short paths. Suppose $\lambda_D(x)$ exists almost everywhere on $\mathcal S$. 
We have 
\begin{equation}\label{eq:lamb_D=bar-lamb_D}
 \bar{\lambda}_D(x)=\lambda_D(x),\qquad \text{a.e.}
\end{equation}
\end{lemma}
%\no We postpone the proof of this lemma because we need the following result first.

\begin{proof}
%[Proof of Lemma \ref{lem:short-paths}] 
We will adapt  the classical argument of Arnold from \cite{Arnold86}.\footnote{One needs to be cautious about this 
argument in the case the domain of the vector field cannot be covered by finitely many flowbox type neighborhoods; see Example 4 in \cite{Vogel03}.} Recall that 
the short curves are denoted by $\{\sigma(x,y)\}_{x,y\in \mathcal S}$.  
The difference of integrals 
in the limits defining $\lambda_D$ \eqref{eq:lambda_D} and $\bar{\lambda}_D$ \eqref{eq:bar-lambda_D} respectively is a sum of 
the following terms (up to permutation of $\int^T_0$ and $\int^1_0$) for $i\geq 1$, and $i+j=k$:
\[
 J_{i,j}=\underbrace{\int^{1}_0\cdots\int^{1}_0}_i\ \underbrace{\int^{T}_0\cdots \int^{T}_0}_j f_{D,X} (\sigma_{x,T}(s_1),\cdots, \sigma_{x,T}(s_i),\phi(x,t_1)\cdots, \phi(x,t_j))\,d\mathbf{t}\,d\mathbf{s},
\]
where $\sigma_{x,T}:[0,1]\to \mathcal S$ is a parametrization of $\sigma(x,\phi(x,T))$. Here $\int^T_0$ is an integral over the orbit of $X$ and $\int^1_0$ is an integral over the short path segment, from \eqref{eq:bar-lambda_D2}.
Fixing small enough $\epsilon>0$, we may subdivide each $[0,T]$ so that the integral $\int^T_0$  is roughly the sum $\int^\epsilon_0+\int^{2\epsilon}_\epsilon+\cdots+\int^{\epsilon\lceil \frac{1}{\epsilon} T\rceil}_{\epsilon(\lceil \frac{1}{\epsilon} T\rceil-1)}$, and each $\epsilon$--interval $[\epsilon\,(k-1),\epsilon\,k]$, $1\leq k\leq \lceil \frac{T}{\epsilon}\rceil$, parametrizes a piece of an orbit within a flowbox chart of $X$. Analogously, we may subdivide the unit intervals parametrizing the short paths and therefore split the $\int^1_0$ integral into the $\epsilon$--pieces, also fitting into flowbox neighborhoods of $X$. Let the index $k_l$, $1\leq l\leq j$, enumerate the sums for the $\int^T_0$'s and the index $m_z$, $1\leq z\leq i$, enumerate the sums for $\int^1_0$'s. Then, the above formula for $J_{i,j}$ yields  
\[
 |J_{i,j}|\leq \sum_{\substack{m_1,\cdots,m_i\\ k_1,\cdots,k_j}} \Bigl|\int^{\epsilon m_1}_{\epsilon (m_1-1)}\cdots\int^{\epsilon m_i}_{\epsilon (m_i-1)}\ \int^{\epsilon k_{1}}_{\epsilon (k_{1}-1)}\cdots \int^{\epsilon k_{j}}_{\epsilon (k_{j}-1)} f_{D,X}\, d\mathbf{t}\,d\mathbf{s}\Bigr|.
\]
Each integral term in the above sum can be expressed, similarly as in \eqref{eq:F-int-w_D-lift}, as an integral over a lift of the product of the short $\epsilon$--pieces of $\sigma$'s and the orbits of $X$, over a smooth differential form $\varpi_D$ on $C[k;\R^3]$. Therefore, by the estimate \eqref{eq:F-vol-est} each integral in the above sum can be bounded above by a constant $A_{X,D}$ which only depends on the vector field, $\varpi_D$, and the metric. Since the number of terms in the sum $J_{i,j}$ is given by $(\lceil 1/\epsilon\rceil)^i (\lceil T/\epsilon\rceil)^{j}$ we obtain
\[
 |J_{i,j}|\leq A_{X,D}(\lceil 1/\epsilon\rceil)^i (\lceil T/\epsilon\rceil)^{j}.
\]
Since there are 
${k\choose i}$ terms of type $J_{i,j}$ in the difference $\bar{\lambda}_D-\lambda_D$ and $j=k-i$, we obtain, for any $x\in \mathcal S$,
\[ 
 |\bar{\lambda}_D(x)-\lambda_D(x)|\leq A_{X,D} \lim_{T\to\infty} \frac{1}{T^k} \sum^k_{i=1} {k\choose i} (\lceil 1/\epsilon\rceil)^i (\lceil T/\epsilon\rceil)^{k-i}=0.\qedhere
\]
\end{proof}
%%%

\subsection{Proofs of Theorems \ref{thm:main-knots}, Corollary \ref{cor:main-cor}, and Theorem \ref{thm:main-ergodic}}

Recall the statements of these three results from the Introduction.
%Theorems \ref{thm:main-knots}, Corollary \ref{cor:main-cor}, and Theorem \ref{thm:main-ergodic} from the Introduction.

\begin{proof}[Proof of Theorem \ref{thm:main-knots} and Corollary \ref{cor:main-cor}]
Part $(i)$ has already been  proven in Key Lemma.
For part $(ii)$, following Theorem \ref{thm:vassiliev}, any finite type $n$ invariant $V_W$ is a linear combination
of integrals $I_D$ of \eqref{eq:I_D(K)}. Specifically, for appropriate coefficients $a_D\in \R$ and $D_1=\includegraphics[scale=.07]{writhe2}$, we can express $V_W$  as  
\begin{equation}\label{eq:T_W-sum}
\begin{split}
 V_W(K) & =\sum^{2n}_{k=1} J_k(K)+b\,I_{D_1}(K),\quad\text{for}\quad J_k(K)=\sum_{D\in TD_n;\ k(D)=k} a_D\,I_D(K),\quad K\in\mathcal{K}.
\end{split}
\end{equation}
\no In order to observe the almost everywhere convergence in 
\[
\overline{\lambda}_{W}(x)=\lim_{T\to \infty} \frac{1}{T^{2n}} V_W(\bar{\mathscr{O}}_T(x)),
\]
we take the corresponding linear combination of $T^{2n}$--time averages of terms in \eqref{eq:T_W-sum}. Namely, we have 
\begin{equation}\label{eq:T_W-time=sum}
 \overline{\lambda}_{W}(x)= \sum^{2n}_{k=1} \lim_{T\to\infty} \frac{1}{T^{2n}} J_k(\bar{\mathscr{O}}_T(x))+b \lim_{T\to\infty} \frac{1}{T^{2n}} I_{D_1}(\bar{\mathscr{O}}_T(x)).
\end{equation}
 By Key Lemma, for $n>1$, the terms in the sum \eqref{eq:T_W-time=sum} with $k<n$ vanish in the limit as does the $I_{D_1}$ term.  As a result, we have 
\[
\overline{\lambda}_{W}(x)=\overline{\lambda}_{W^{n}}(x)=\sum_{D\in TD_n; k
(D)=2n} a_D\,\overline{\lambda}_{D}(x).
\]  
Further, if $J_{2n}(\bar{\mathscr{O}}_T(x))$ is $o(T^{2n-1})$, then we may consider $T^{2n-1}$--time averages of 
$V_W$ and obtain
\[
\overline{\lambda}_{W}(x)=\overline{\lambda}_{W^{2n-1}}(x)=\sum_{D\in TD_n;\ k(D)=2n-1} a_D\,\overline{\lambda}_{D}(x).
\]
This reasoning further applies, if the terms $J_k$ are of lower order, and this therefore gives the proof of Corollary \ref{cor:main-cor}. 

It remains to prove invariance under volume-preserving deformations as claimed in  $(iii)$. Note that, given $h\in\text{Diff}_0(\R^3,\mu)$, the short path system $h\Sigma=\{h\circ\sigma\}$ on $\mathcal{S}'=h(\mathcal{S})$ obtained from $\Sigma=\{\sigma\}$ has the same properties as the original system $\Sigma$ on $\mathcal{S}$ with respect to the pulled-back metric on $\mathcal{S}'$. In particular, Lemma \ref{lem:short-paths} holds for $h\Sigma$. Now, for any $T>0$ and $x\in\mathcal{S}$, consider knots $K_{h_\ast X}=\bar{\mathscr{O}}^{h_\ast X}_{T}(x)$ and  $K_{X}=\bar{\mathscr{O}}^{X}_{T}(h^{-1}(x))$ (where we used $h\Sigma$ to close up $K_{h_\ast X}$ and $\Sigma$ to close up $K_{X}$). By \eqref{eq:flow-conj}, we have
\[
 K_{h_\ast X}=h(K_{X}).
\]
Since $h\in \text{Diff}_0(\R^3,\mu)$, we conclude that $K_{h_\ast X}$ and $K_X$ are isotopic, implying 
\[
 V_W(K_{h_\ast X})=V_W(K_{X}).
\] 
Taking $\lim_{T\to\infty} \frac{1}{T^{2n}} \bigl(\ \cdot\ \bigr)$ of both sides in the above equation yields
\[
 \overline{\lambda}^{h_\ast X}_W(x)=\overline{\lambda}^{X}_W(h^{-1}(x)),\qquad a.e.
\]
After a change of variables (using the fact that $h$ is $\mu$-preserving), we obtain 
\[
\mathscr{V}_W(h_\ast X)=\mathscr{V}_W(X). \qedhere
\]
\end{proof}
%%%%%%%%%%%%%%%%%%%%%%%%%%%%%%%
From the above argument, observe that 
$\lambda_{W^k}(x)$ is a time average of the $L^1$--functions
\[
 f_{W,X,k}:C(2n;\R^3)\longrightarrow \R,\qquad f_{W,X,k}:=\sum_{D\in TD_n; k(D)=k} a_D\,f_{D,X,k},
\] 
for $f_{D,X,k}=f_{D,X}$ as defined in \eqref{eq:f_D,X}. Applying the Multiparameter Ergodic Theorem \cite{Becker81, Tempelman67} to $f_{W,X,k}$, we obtain the following formula 
for the vector field invariant $\mathscr{V}_{W,k}:\operatorname{Vect}(\mathcal{S},\mu)\longrightarrow \R$: 
\begin{equation}\label{eq:higher-helicity}
 \mathscr{V}_{W,k}(X)=\int_{\mathcal{S}^{k}} f_{W,X,k}\, \bar{\mu}_\Delta
\end{equation}
(recall $\bar{\mu}_\Delta$  is a diagonal invariant measure on $\mathcal{S}^{k}$ given in
\eqref{eq:bar-mu-delta}). 
%%%%%%%%%%%%%%%%%%%%%%%%%%%%%%%%%%%%%
\begin{figure}[ht]
\begin{center}
\includegraphics[scale=0.5]{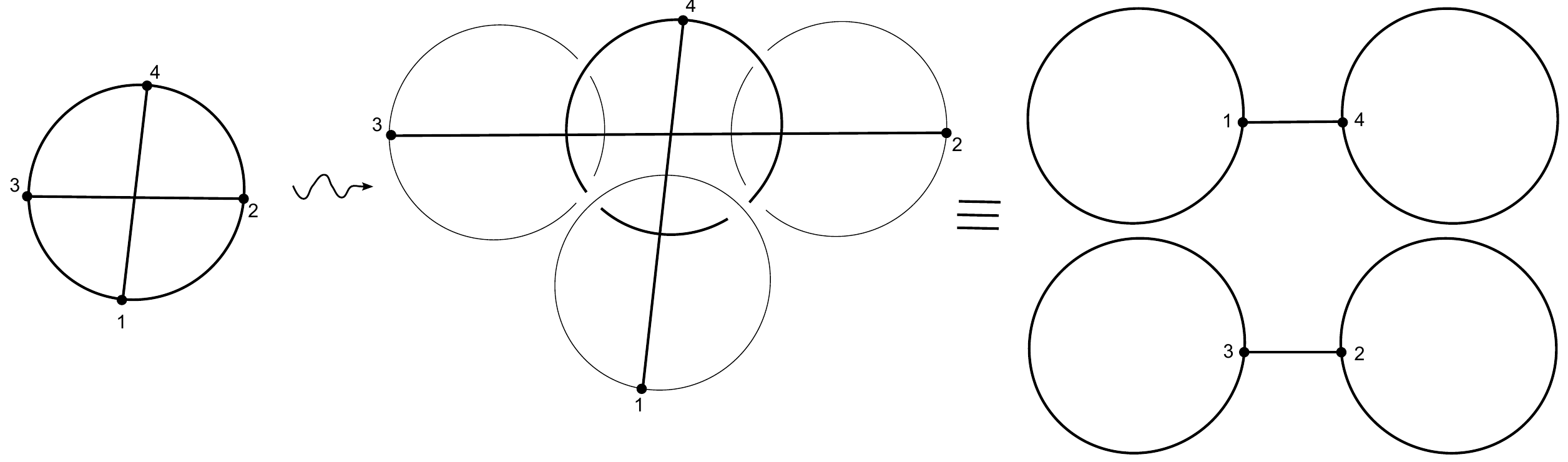}
\caption{ {\small A top degree diagram perturbation leads to pairwise linking number diagrams.}}
\label{fig:diag-p}
\end{center}
\end{figure} 
%%%%%%%%%%%%%%%%%%%%%%%%%%%%%%%%%%%%%
\begin{proof}[Proof of Theorem \ref{thm:main-ergodic}] 
By assumption, the domain $\mathcal{S}$ is equipped with the standard volume form $\mu$ and $X$ is an
ergodic $\mu$--preserving nonvanishing vector field. For simplicity, we assume that $\mu$ induces a probability measure on $\mathcal{S}$. Ergodicity of $X$ on $\mathcal{S}$ implies, among other things, that almost every orbit of $X$ densely fills the interior of $\mathcal{S}$.  Clearly, $\mu$ induces a $\phi^k_X$--invariant ergodic probability measure on $\mathcal{S}^{k}$ via the $3k$--form $\mu^k=\underbracket[0.5pt]{\mu\times\cdots\times\mu}_{k\text{ times}}$. 
 By Key Lemma, $f_{W,X}$ is in $L^1(\mu^{2n})$, and thus the ergodicity of the $\phi^{2n}_X$--action implies that the integral 
 \[
 \int_{\mathcal{S}^{2n}} f_{W,X} \mu^{2n}
 \]
equals
\begin{equation}\label{eq:ergodic-time-avg}
 \lim_{T\to \infty} \frac{1}{T^{2n}} \underbracket[0.5pt]{\int^{T}_0\cdots\int^{T}_0}_{2n\text{ times}} f_{W^{2n},X}
(\phi(x_1,t_1),\cdots, \phi(x_{2n},t_{2n}))\,dt_1\cdots dt_{2n},
\end{equation}
for almost every point $\mathbf{x}=(x_1,\cdots,x_{2n})\in \mathcal{S}^{2n}$.
Choosing $\mathbf{x}$ to be away from the thick diagonal, we have $2n$ distinct orbits $\bar{\mathscr{O}}_T(\mathbf{x})=\bar{\mathscr{O}}_T(x_1)\times\cdots\times \bar{\mathscr{O}}_T(x_{2n})$ through each coordinate point. For each top degree diagram (i.e.~a chord diagram) $D\in \mathcal{D}_n$, $k(D)=2n$, the integral of the associated differential form $\varpi_D$ over $\bar{\mathscr{O}}_T(\mathbf{x})$
splits as a product of linking numbers of pairs of points associated with the chords of $D$. This can be thought of as a perturbation of the diagram, as the vertices are no longer on the same orbit; see Figure \ref{fig:diag-p} for an illustration. Explicitly, for $\bar{\mathscr{O}}_T(\mathbf{x})$ and $\varpi_D=\prod_{(i,j)\in \mathcal{E}(D)} \omega_{i,j}$, from \eqref{eq:f_D,X} and the fact that $\int_{\bar{\mathscr{O}}_T(x_i)\times \bar{\mathscr{O}}_T(x_j)} \omega_{i,j}=\text{lk}(\bar{\mathscr{O}}_T(x_i),\bar{\mathscr{O}}_T(x_j))$, we have (up to short paths)
\begin{equation}\label{eq:int-f_D-lk}
\underbracket[0.5pt]{\int^{T}_0\cdots\int^{T}_0}_{2n\text{ times}} f_{D,X}
(\phi(x_1,t_1),\cdots, \phi(x_{2n},t_{2n}))\,dt_1\cdots dt_{2n}=\prod_{(i,j)\in \mathcal{E}(D)} \text{lk}(\bar{\mathscr{O}}_T(x_i),\bar{\mathscr{O}}_T(x_j)).
\end{equation}
By definition of $\mathscr{H}(X)$ (see \eqref{eq:helicity}) and the ergodicity assumption,
summing up over all top order diagrams $D\in\mathcal{D}_n$, we obtain from \eqref{eq:ergodic-time-avg} the independence of the limit of short paths and from \eqref{eq:int-f_D-lk} we obtain
\begin{equation}\label{eq:c_W*helicity}
 \int_{\mathcal{S}^{2n}} f_{W,X} \mu^{2n}=c_{W} (\mathscr{H}(X))^n,
\end{equation}
where $c_{W}$ is a constant independent of $X$.

 Next, we turn to the proof of the identity in 
\eqref{eq:T_W=helicity}. Observe that in the space of probability measures $\mathcal{M}(\mathcal{S}^{2n})$, the diagonal measure $\mu_{\Delta}$ can be approximated by a sequence of probability measures supported on the $\delta$--tubular neighborhood $U_\delta=U_\delta(\Delta)$ of the thin diagonal $\Delta$ of 
$\mathcal{S}^{2n}$. These measures can be precisely defined as 
\[
 \nu^{2n}_\delta=\frac{\chi_\delta}{\text{vol}(U_\delta)} \,\nu^{2n},\qquad \chi_\delta(\mathbf{x})=\begin{cases}
  1, & \quad \mathbf{x}\in U_\delta,\\
  0, & \quad \mathbf{x}\not\in U_\delta.
 \end{cases}
\]
Since $\mu^{2n}_\delta\to \mu_{\Delta}$, $\delta\to 0$ in $\mathcal{M}(\mathcal{S}^{2n})$. Thanks to the weak compactness of $\mathcal{M}(\mathcal{S}^{2n})$, the sequence of the associated invariant measures $\overline{\mu}^{2n}_\delta$, built via the formula  \eqref{eq:bar-mu-delta}, converges to the diagonal invariant measure $\overline{\mu}_\Delta$ in $\mathcal{M}(\mathcal{S}^{2n})$.
From Key Lemma, for each $\delta$, $f_{W,X}$ is in $L^1(\overline{\mu}^{2n}_\delta)$.
Since the right hand side in \eqref{eq:ergodic-time-avg} is independent of the choice of $\mathbf{x}$ (as long as it is generic), for a given $\delta$ we may suppose $\mathbf{x}\in U_\delta$ and obtain from \eqref{eq:c_W*helicity} and the assumption of ergodicity  the identity
\[ 
\int_{\mathcal{S}^{2n}} f_{W,X} \overline{\mu}^{2n}_\delta=\lim_{T\to \infty} \frac{1}{T^{2n}} \underbracket[0.5pt]{\int^{T}_0\cdots\int^{T}_0}_{2n\text{ times}} f_{W,X}
(\mathbf{x},\mathbf{t})\,d\mathbf{t}=c_{W} (\mathscr{H}(X))^n.
\]
Since $\overline{\mu}^{2n}_\delta\to \overline{\mu}_{\Delta}$ in $\mathcal{M}(\mathcal{S}^{2n})$, we deduce \eqref{eq:T_W=helicity}. The second part of Theorem \ref{thm:main-ergodic} can be justified analogously.
\end{proof}

%%%%%%%%%%%%%%%%%%%%%%%%%%%%%%%%%%%%%%%%%%%%%%%%%%
\section{Quadratic helicity, energy, and proof of Theorem \ref{thm:q-helicity-energy}.}\label{sec:q-helicity}
The methods presented in the previous sections can be applied almost without any changes to the setting of asymptotic links. One difference between the case of knots and links is a choice of the diagonal invariant measure $\overline{\mu}_\Delta$
in \eqref{eq:bar-mu-delta}. Rather than presenting this obvious generalization, the rest of this section is devoted to an illustration of the relevant constructions for the simplest finite type $2$ invariant associated with a $2$--component link, the square of the linking number $\operatorname{lk}^2$. We observe that in the setting of asymptotic links, $\operatorname{lk}^2$ leads to {\em quadratic helicity} that was recently proposed by Akhmetiev in \cite{Akhmetev12}. Further,  it is the simplest invariant that can provide a sharper lower bound for the fluid energy than $\mathscr{H}(X)$, as claimed in Theorem \ref{thm:q-helicity-energy}.

The weight system associated to $\operatorname{lk}^2$ is given by just one trivalent diagram which we denote by $D_{\operatorname{lk}^2}$, pictured in Figure \ref{fig:q-linking}. 
%%%%%%%%%%%%%%%%%%%%%%%%%%%%%%%%%%%%%%%%%%%%
 \begin{figure}[htbp]
\begin{center}
 \includegraphics[scale=.8]{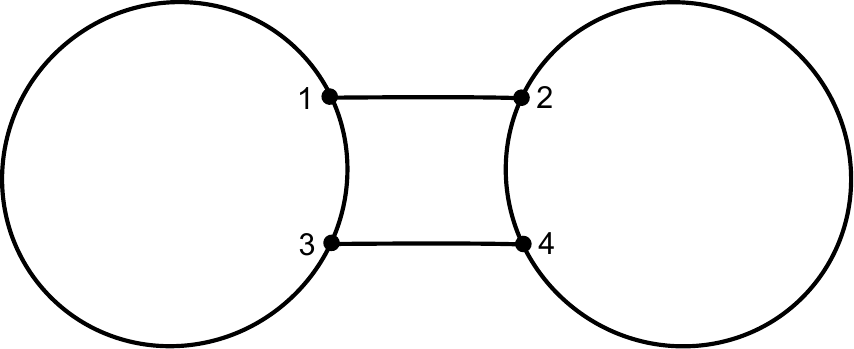}
\caption{A trivalent diagram $D_{\operatorname{lk}^2}$ for $\operatorname{lk}^2$.}
\label{fig:q-linking}
\end{center}
\end{figure}  
The configuration of points and chords on $D_{\operatorname{lk}^2}$ implies a choice of the invariant measure on $\mathcal{S}^4$ associated with the flow of $X$.
Namely, we start with the product $\phi^2_X=\phi_X\times\phi_X$--invariant measure $\mu\times\mu$ on $\mathcal{S}\times\mathcal{S}$ and push it forward to the $4$--fold product $\mathcal{S}^4$ by the inclusion $j\colon(x,y)\longmapsto (x,x,y,y)$. Let us denote the diagonal, parametrized by $j$, by $\Delta_{(2)}=\Delta_{\{\{1,3\},\{2,4\}\}}$. Also denote the pushforward measure by $\mu_{\Delta_{(2)}}$ and the associated $\phi^4_X$--invariant measure by $\overline{\mu}_{\Delta_{(2)}}$ (i.e.~$\overline{\mu}_{\Delta_{(2)}}=\overline{\int}\mu_{\Delta_{(2)}}$).
By virtue of Theorem \ref{thm:main-knots}, the asymptotic invariant of $X$ associated with $\operatorname{lk}^2$ equals the quadratic helicity of \cite{Akhmetev12} and is by \eqref{eq:higher-helicity} given as 
\begin{equation}\label{eq:q-helicity}
 \mathscr{H}^2(X)=\int_{\mathcal{S}^{4}} \varpi_{D_{\operatorname{lk}^2}}(X,X,X,X)\,\overline{\mu}_{\Delta_{(2)}},
\end{equation}
where 
\[
\varpi_{D_{\operatorname{lk}^2}}=\alpha^\ast\omega_{1,2}\wedge\alpha^\ast\omega_{3,4}.
\]
(because $D_{\operatorname{lk}^2}$ has no
free vertices).
 Observe that $\mathscr{H}^2(X)\geq 0$, whereas $\mathscr{H}(X)$ can be negative. We can easily show examples when $\mathscr{H}(X)=0$ but $\mathscr{H}^2(X)>0$ (see \cite[p.~344]{Arnold86}). Therefore it is of  general interest to derive an analog of inequality \eqref{eq:E-3/2-helicity-bound} for $\mathscr{H}^2(X)$.  

\begin{proof}[Proof of Theorem \ref{thm:q-helicity-energy}]
Recall \cite{Candel-Conlon00} that the diagonal invariant measure $\overline{\mu}_{\Delta_{(2)}}$ can be arbitrarily well approximated in $\mathcal{M}(\mathcal{S}^4)$ by positive finite linear combinations 
\begin{equation}\label{eq:bar-mu-n}
 \overline{\mu}_n=\sum^n_{i=1} a_i\,\overline{\mu}_{{\bf x}_i},\qquad a_i>0,
\end{equation}
where $\overline{\mu}_{{\bf x}_i}$ is a $\phi^4_X$--invariant measure obtained from averaging a Dirac delta $\delta_{{\bf x}_i}$ supported at a point ${\bf x}_i=(x_i,x_i,y_i,y_i)$ on the diagonal $\Delta_{(2)}$. More precisely, if
$\mu_n=\sum^n_{i=1} a_i\,\delta_{{\bf x}_i}$ as an approximation of $\mu_{\Delta_{(2)}}$, $\overline{\mu}_n=\overline{\int}\mu_n$ is an approximation of $\overline{\mu}_{\Delta_{(2)}}$. In fact, approximating $\mu\times\mu$ by 
$\sum^n_{i=1} b_i\,\delta_{(x_i,y_i)}$, $b_i>0$, and applying the pushforward under $j$ we conclude that the coefficients in \eqref{eq:bar-mu-n} are given as 
\[
 a_i=b_i^2.
\]
Note that each $\overline{\mu}_{{\bf x}_i}$ is a product measure, i.e. 
\begin{equation}\label{eq:bar-mu_i-prod}
 \overline{\mu}_{{\bf x}_i}=\overline{\mu}^{\{1,2\}}_{(x_i,y_i)}\times\overline{\mu}^{\{3,4\}}_{(x_i,y_i)},
\end{equation}
where $\overline{\mu}^{\{k,l\}}_{(x_i,y_i)}$ is a pushforward of $\overline{\mu}_{(x_i,y_i)}=\overline{\int} \delta_{(x_i,y_i)}$ under the inclusion 
of $\mathcal{S}\times \mathcal{S}$ into the $(k, l)$-coordinates factor of $(\mathcal{S}\times \mathcal{S})^2=\mathcal{S}^4$.
By the proof of Theorem \ref{thm:main-knots}, the function $f_W=\alpha^\ast\omega_{1,2}\wedge\alpha^\ast\omega_{3,4}(X^{\wedge^4})$ is $\overline{\mu}_n$--integrable for each $n$. Moreover, if we set 
\[
 f_{1,2}=\alpha^\ast\omega_{1,2}(X,X), \quad  f_{3,4}=\alpha^\ast\omega_{3,4}(X,X),
\]
then
\[
 f_W=\alpha^\ast\omega_{1,2}\wedge\alpha^\ast\omega_{3,4}(X^{\wedge^4})  =\alpha^\ast\omega_{1,2}(X,X)\,\alpha^\ast\omega_{3,4}(X,X)=f_{1,2}\,f_{3,4}.
\]
Note that the functions $f_{1,2}$ and $f_{3,4}$ are constant on appropriate $\mathcal{S}^2$ factors of $\mathcal{S}^4$.
Using \eqref{eq:bar-mu-n} and \eqref{eq:bar-mu_i-prod}, we obtain
\begin{align*}
 \Bigl|\int_{\mathcal{S}^4} f_W\,\overline{\mu}_n\Bigr| & =\Bigl|\sum^n_{i=1} b^2_i\Bigl(\int_{\mathcal{S}^2} f_{1,2}\,\overline{\mu}^{\{1,2\}}_{(x_i,y_i)}\Bigr)\Bigl(\int_{\mathcal{S}^2} f_{3,4}\,\overline{\mu}^{\{3,4\}}_{(x_i,y_i)}\Bigr)\Bigl|
 \end{align*}
\begin{align*}
 & \leq \Bigl(\sum^n_{i=1} b_i\int_{\mathcal{S}^2} |f_{1,2}|\,\overline{\mu}^{\{1,2\}}_{(x_i,y_i)}\Bigr)\Bigl(\sum^n_{i=1} b_i\int_{\mathcal{S}^2} |f_{3,4}|\,\overline{\mu}^{\{3,4\}}_{(x_i,y_i)}\Bigr).
\end{align*}
Passing to the limit in $\mathcal{M}(\mathcal{S}^4)$ as $n\to \infty$, we have $\mu_n\to \overline{\mu}_{\Delta_{(2)}}$ and $\sum_i b_i\overline{\mu}^{\{k,l\}}_{(x_i,y_i)}\to \mu\times\mu$. Therefore  
\begin{equation}\label{eq:H^2-leq-cross}
\mathscr{H}^2(X) \leq \Bigl(\int_{\mathcal{S}^2} |\alpha^\ast\omega_{1,2}(X,X)|\,\mu\times\mu\Bigr)\Bigl(\int_{\mathcal{S}^2} |\alpha^\ast\omega_{3,4}(X,X)|\,\mu\times\mu\Bigr)=c(X)^2,
\end{equation}
where $c(X)$ stands for the asymptotic crossing number as given in \cite[p.~191]{Freedman-He91}. The estimate \cite[Equation (1.9)]{Freedman-He91}  
\[
 E_{3/2}(X)\geq \Bigl(\frac{16}{\pi}\Bigr)^{1/4} c(X)^{3/4}
\]
immediately yields the required bound in \eqref{eq:E_3/2-q-helicity}.
\end{proof}
%%%%%%%%%%%%%%%%%%%%%%%%%%%%%%%%%%%%%%%%%%%%%%%%%%%%%%%%%%%%%%%%%%%%%%%%%%%%%

\end{document}